\documentclass[11pt]{article}
\usepackage{hyperref}
\usepackage[T1]{fontenc}
\usepackage{amsfonts,eufrak}
\usepackage[english]{babel}
\usepackage{a4wide,times}

\usepackage[latin1]{inputenc}
\usepackage{amssymb}
\usepackage{epsfig}
\usepackage{upgreek}
\usepackage{amsbsy}
\usepackage{verbatim}
\usepackage{color}
\usepackage{mathrsfs}
\usepackage{graphicx}
\usepackage{epstopdf}
\usepackage{makeidx}
\usepackage{amsmath,amsbsy}
\usepackage{mathabx}
\usepackage{amsthm}
\usepackage{fleqn}
\usepackage{xargs}
\usepackage[capitalise]{cleveref}
\usepackage[normalem]{ulem}
\usepackage[prependcaption]{todonotes}
\newcommandx{\fab}[2][1=]{\todo[inline, author={Fab}, linecolor=green,backgroundcolor=red!25,bordercolor=red,#1]{#2}}
\newcommandx{\fabnote}[2][1=]{\todo[author={Fab}, linecolor=red,backgroundcolor=red!25,bordercolor=red,#1]{#2}}

\theoremstyle{plain}
\newtheorem{thm}{Theorem}[section]
\newtheorem{cor}[thm]{Corollary}
\newtheorem{lem}[thm]{Lemma}
\newtheorem{prop}[thm]{Proposition}

\theoremstyle{definition}

\usepackage{dsfont}

\theoremstyle{remark}
\newtheorem{rem}{\bf Remark}[section]
\theoremstyle{remark}

\newtheorem{com*}{\bf Comment}
\usepackage{fancyhdr}

\makeatletter
\def \newequation#1#2{
   \@definecounter{#1}
   \@namedef{the#1}{\hbox{#2}}
   \@namedef{#1}{$$\refstepcounter{#1}}
   \@namedef{end#1}{
      \eqno \csname the#1\endcsname $$\global\@ignoretrue
      }
}
\makeatother
\newequation{E1}{($E_{b,\sigma,W}$)}
   
\makeatletter
\def \newequation#1#2{
   \@definecounter{#1}
   \@namedef{the#1}{\hbox{#2}}
   \@namedef{#1}{$$\refstepcounter{#1}}
   \@namedef{end#1}{
      \eqno \csname the#1\endcsname $$\global\@ignoretrue
      }
   }
\makeatother
\newequation{hyp3}{($\mathcal{H}_{b,\sigma}$)}

\makeatletter
\def \newequation#1#2{
   \@definecounter{#1}
   \@namedef{the#1}{\hbox{#2}}
   \@namedef{#1}{$$\refstepcounter{#1}}
   \@namedef{end#1}{
      \eqno \csname the#1\endcsname $$\global\@ignoretrue
      }
   }
\makeatother
\newequation{shleg}{(ShLeg)}

\makeatletter
\def \newequation#1#2{
   \@definecounter{#1}
   \@namedef{the#1}{\hbox{#2}}
   \@namedef{#1}{$$\refstepcounter{#1}}
   \@namedef{end#1}{
      \eqno \csname the#1\endcsname $$\global\@ignoretrue
      }
   }
\makeatother
\newequation{kl}{(KL)}

\makeatletter
\def \newequation#1#2{
   \@definecounter{#1}
   \@namedef{the#1}{\hbox{#2}}
   \@namedef{#1}{$$\refstepcounter{#1}}
   \@namedef{end#1}{
      \eqno \csname the#1\endcsname $$\global\@ignoretrue
      }
   }
\makeatother
\newequation{haar}{(Haar)}

\newcommand{\ds}{\displaystyle}  
\newcommand{\Cloc}{\ensuremath{\mathcal{C}_{\mathrm{loc}}}}
\def\HFBM{\mathbf{(C_{H,\zeta})}}
\def\ATROIS{\mathbf{C}_{{\rm TV}}(\gamma)}
\def\ATROISbis{(\mathbf{C}_{{\cal W}_1})}
\def\ATROISter{(\mathbf{C}_{{\cal W}_1}')}
\def\AUN{(\mathbf{A}_{\beta,q,\frac{1}{2}})}
\def\AUNgamma{(\mathbf{A}_{\beta,q,\gamma})}

\def\ADEUX{\mathbf{(P)}}
\def\Exp{\mathbf{E}}
\def\ee{e}
\def\proba{\pi}
\def\tzero{{t_0}}
\def\R{{\mathbb{R}}}
\def\ER{{\mathbb{R}}}

\def\N{{\mathbb{N}}}
\def\E{{\mathbf{E}}}

\def\L{{\cal L}}

\def\B{{\cal B}}
\def\P{{\mathbb{P}}}
\def\PE{{\mathbb{P}}}

\def\F{{\cal F}}
\def\W{{\cal W}}

\def\a{\alpha}
\def\b{\beta}
\def\g{\gamma}
\def\d{\delta}

\def\s{\sigma}
\def\ve{\varepsilon}

\def\Mpiq{\mathsf{M}_{\pi,q}}

\def\LSF{{\bf (LSF)}}

\author{Gilles Pag\`es~\thanks{Sorbonne Universit\'e,  Laboratoire de Probabilit\'es, Statistique et Mod\'elisation, UMR~8001, case 158, 4, pl. Jussieu, F-75252 Paris Cedex 5, France. E-mail: \texttt{gilles.pages@sorbonne-universite.fr}}  $\;$and 
Fabien Panloup~\thanks{Univ Angers, CNRS, LAREMA, SFR MATHSTIC, F-49000 Angers, France. E-mail: \texttt{fabien.panloup@univ-angers.fr}}
}
\title{\bf Convergence rate  of the occupation measure of { classes of ergodic processes toward their invariant distribution in mean Wasserstein distance}}

\begin{document}
\maketitle
%\todo{Dans la preuve du Corollaire 2.4, l'quivalence avec (2.17) de Wang utilise que $b$ et $\sigma$ sont Lipschitz : le mentionner}

\begin{abstract}
In~\cite{FG}, Fournier and Guillin obtained some bounds of the $L^p$-mean rate of convergence in Wasserstein distance of empirical distributions for a class of stationary mixing processes. 
In this paper, we propose to extend their strategy of proof and provide general criterions which allow to keep similar rates for a larger class of processes. These results (which do not require regularization techniques) lead to various applications to  occupation measures of ergodic processes which may be not stationary or not Markovian under an assumption of \textit{conditional} convergence to equilibrium in Total Variation or Wasserstein distance. We then provide explicit conditions which lead to these rates for Brownian diffusions and additive SDEs driven by fractional Brownian Motions {or by Gaussian processes with stationary increments}.
\end{abstract}

%68W40; 60H30; 62L20; 60J60}


%regularit\'e de $b$ et $\s$ ds le thms principaux \`a v\'erifier\, 
%Gestion de la section Confluence r\'esiduelle:   $\HTVO$, $\HWO$ vs $H^0_{\rho}$, 
%question de $\s$ $d\times q$ (a priori ok) ou $d\times d$.}
%
%\noindent 
%{-- Inserer resultat sur la cvgce forte du schema \`a $\s$ non constant (ne servira pas) et a $\s$ constant qui servira.} 
%\%section{Introduction}
\section{Introduction}

%{2/Questions principales \'evoqu\'ees avec Gilles ce vendredi : Voir \'egalement dans l'article de Wang problème ou point à comprendre concernant la factorisation (dans le cas $q<2$) car on divise par $|.|^{2-q}$...}

%Our aim is to emphasize that the methods developed in~\cite{Dereich13} and~\cite{FG} to provide non-asymptotic bound for the Wasserstein distance ${\cal W}_p(\mu,\nu)$ between two probability measures with finite $p$-moments and more specifically between the occupation measure of a sample of i.i.d. or stationary  $??$-mixing random vectors $L^p$-integrable  can be extended and applied to other fields of application in continuous time to evaluate how fast the occupation measure of an ergodic process converges toward its (marginal) invariant distribution in Markovian and non-markovian frameworks, starting or not from its stationary regime. For other extensions, in connection with discrepancies, Kolmogorov-Smirnov distance and their applications to quasi-Monte Carlo simulation, we refer to another recent contribution~\cite{PanPagDiscrep}. \dots
\medskip
%In this paper we will extend the fields of application of Fournier-Guillin's Theorem~1 in~\cite{FG}) on the $L^p$-mean rate of convergence in Wasserstein distances  in the Strong Law of Large Numbers (SLLN).
In this paper we will extend the fields of application of Fournier-Guillin's Theorem~1 in~\cite{FG}) on the $L^p$-mean rate of convergence in Wasserstein distances  in the Strong Law of Large Numbers (SLLN).

% i.e.  the empirical  measure of a $\nu$-distributed i.i.d. sample of  random vectors to a continuous time setting. 
Our objective is to establish using the same approach similar results in continuous time for various classes of right continuous left limited (a.k.a. c\`adl\`ag) mean-reverting or contracting ergodic processes by estimating the rate of convergence of their occupation measure  which  $a.s.$ converges toward their (marginal) invariant distribution, still in $L^p$-Wasserstein distance. These processes can be Markovian or not, at equilibrium (i.e. stationary) or not. 

% i.e.  the empirical  measure of a $\nu$-distributed i.i.d. sample of  random vectors to a continuous time setting. 

Let us recall that, for a given norm $|\,.\,|$ on $\ER^d$, when $p\ge 1$, the $p$-Wasserstein distance is defined  for every $\mu$, $\nu\!\in{\cal P}_p(\R^d)$, the space of probability distributions on ${\mathcal B}or(\R^d)$ (Borel sets of $\R^d$) having (at least) $p$-finite moments, by
$$
{\cal W}_p^{|\,.\,|}(\mu,\nu)=\inf \left\{ \left[\int_{\R^d} \hskip-0,25cm |x-y|^p\pi(dx,dy)\right]^\frac{1}{p}\!\!,\, \pi \!\in {\cal P}_{\!p}(\R^d\times \R^d), \pi(dx\times\R^d)=\mu, \, \pi( \R^d\times dy)=\nu \right\}\!,
$$
(when $p\,{\ge1}$, {${\cal W}_p^{|\,.\,|}$} is a distance.) In the sequel, we will only write ${\cal W}_p$ to alleviate notation.
It is well-known that ${\cal P}_p(\ER^d)$  equipped with  ${\cal W}_p$ is  a Polish space on which convergence of sequences is characterized by weak convergence combined with convergence of  {$p$-moments}  (see $e.g.$~\cite{Bolley08} or~\cite{Villani} for details). The $p$-Wasserstein distance is now widely used in probabilistic and statistical applications. In statistics, this distance usually produces a robust alternative to Kullback-Leibler divergence taking into account the underlying metric structure. In probability theory, the Wasserstein distance is also widely used for quantifying the rate of convergence to equilibrium or analyzing the robustness of stochastic algorithms.

The starting point   in~\cite{FG} --~which partially relies on ideas from a former work~\cite{Dereich13}~--  is to establish general upper-bounds of the $p$-Wasserstein distance ${\cal W}_p(\mu,\nu)$ between two probability distributions $\mu$ and $\nu\!\in{\cal P}_q(\R^d)$ {for} $q>p$, by means of a telescopic splitting of the two measures based on nested refined partitions of hypercubes.   In~\cite{Dereich13}, the first objective of the authors was an application to the optimal mean rate  random quantization (a new proof of the so-called  Pierce  Lemma)  but the paper  also provides a  (partial) result on the mean $p$-Wasserstein rate of convergence  in the {SLLN}. In~\cite{FG} a complete analysis of this question is carried out. This paper also applies their approach  to evaluate in a discrete  time setting the $p$-Wasserstein convergence rate of the (random) empirical measure $\nu_n(\omega,d\xi):= \frac 1n \sum_{k=0}^{n-1}\delta_{X_k(\omega)}{(d\xi)}$ toward  the marginal invariant distribution  $\nu$  of some classes of stationary mixing sequences.  {Up to a logarithmic term when $p=d/2$, the main result of  \cite{FG} reads
\begin{equation}\label{theo:FGres} \|{\cal W}_p(\nu_n,\nu)\|_p\lesssim n^{-\left(\frac{1}{d}\wedge \frac{1}{2p}\right)},
\end{equation}
where for a random variable $X$, $\|X\|_p=\E[|X|^p]^{\frac{1}{p}}$.} All these approaches are based on the $L^1$ --~or $L^2(\P)$~-- convergence rate of $\nu_n(\omega,A)$ toward $\nu(A)$ along the whole class of hypercubes   $A\subset\R^d$ obtained by dilatation-translation form the unit hypercube.

The main contribution of  this paper is not  to switch to continuous time which is essentially straightforward but rather to emphasize that the field of application of this ``methodology (or ``toolbox")  turns out {to be} much wider than  the striking Theorem~1 on the SLLN in~\cite{FG} {thanks  to several} natural {extensions that we propose} {(see {e.g.}  \cref{prop:fund} and \cref{thm:abstrait})} . {As a consequence, which was our initial objective,} {it is in particular possible to move beyond Poincar\'e or $\rho$-mixing assumptions that restrict applications even within a Markovian framework}. Doing so, we will be able to  apply {this extended approach}  {not only to various} {Markovian but also  non-Markovian} processes sharing some mean-reverting and/or contraction properties (which implies that they have, {at least  in some sense in a non-Markovian framework,} a stationary regime) at {equilibrium but {\it also out of}  equilibrium}.

More precisely, our paper is mainly structured around the following types of general statements:

\noindent  {$\bullet$ For an $({\cal F}_t)$-adapted  process $(X_t)_{t\ge0}$ (not necessarily stationary) with occupation measure{s} $\nu_t(\omega,d\xi)= \frac 1t \int_0^t \delta_{X_s(\omega)}ds$, {$t>0$,} a slight relaxation of the tools of \cite{FG} allows to get some bounds for $\|{\cal W}_p(\nu_t,\nu)\|_p$ as soon as}
%\noindent $\blacktriangleright$ it is possible to \textit{almost} preserve the Wasserstein rates of \eqref{theo:FGres} for the occupation measure $\nu_t(\omega,d\xi)= \frac 1t \int_0^t \delta_{X_s(\omega)}ds$ of an $({\cal F}_t)$-adapted  process $(X_t)_{t\ge0}$ as soon as
{\begin{itemize}
\item[$\blacktriangleright$] the conditional distributions  $({\cal L}(X_t|{\cal F}_s))_{t\ge s}$  converge to equilibrium in total variation\footnote{Wasserstein criterions are also provided.}, uniformly in $s$. 
\item[$\blacktriangleright$] Some appropriate uniform controls hold for the moments of the process.
\end{itemize}
$\bullet$ The bounds for $\|{\cal W}_p(\nu_t,\nu)\|_p$ are mainly governed by the rate of convergence in the first condition. In particular, when the rate is integrable  and the controls of the moments hold for sufficiently large order $q$, then, one \textit{almost} retrieves the rate of  \eqref{theo:FGres}, \textit{i.e.}
\[
\big\| {\cal W}_p(\nu_t,\nu) \big\|_p\le C_{p,d,\varepsilon} t^{-\left(\frac{1}{2p}\wedge\left(\frac{1}{d}-\ve\right)\right)} \quad \mbox{ for every small enough $\ve>0$}.
\]
for some real constant $C_{p,d,\varepsilon}\!\in (0, +\infty)$.}

%The field of applications could be resumed as follows: for a càdlàg process $(X_t)_{t\ge0}$ with 

% with occupation measure $\nu_t(\omega,d\xi)= \frac 1t \int_0^t \delta_{X_s(\omega)}ds$ and invariant distribution $\nu$ (which needs
%{@G: ne devrait-on pas mettre un pseudo-résultat ici ? Par exemple, In this paper, we are going to show that under  \textit{an assumption of conditional convergence to equilibrium} and sufficient control of the moments of the process, we are able to retrieve almost the same rate as \eqref{theo:FGres} for the occupation measure  }

%{@G: Je trouve qu'on digresse un peu trop dans ce qui suit.}
%{This result may for instance allow to replace 

{ These general results then allow to consider many examples  in the class {of}  diffusion processes  solutions to Brownian SDEs and  more generally  to  various classes of Feller Markov processes, at least when the invariant distribution $\nu$ is unique (without Poincar\'e-type assumptions). They may also apply to non-Markovian processes such as SDEs driven by fractional Brownian motion or asymptotically homogeneous Markov processes. Some of these examples will be studied in details in the paper.}

%For instance, typical examples can be found in the class {of}  diffusions processes  solutions to Brownian SDEs and  more generally  to  various classes of Feller Markov processes, at least when the invariant distribution $\nu$ is unique. In fact the  control of the  $\mathcal{W}_p$ -rate of convergence of the occupation measure $\nu_t(\omega,d\xi)= \frac 1t \int_0^t \delta_{X_s(\omega)}ds$ toward $\nu$  {is} important for applications {beyond the case of standard ergodic diffusions}.

 So far, to control such a rate of convergence {in some general settings}, {a usual technique was to regularize} both the occupation measure and $\nu$ by a Gaussian noise (\cite{WangFY24,ChassPag25}) which automatically produce sub-optimal rates due to the exogenous noise introduced into the  problem. Thus, among others, on may mention the analysis of the mean convergence rate of the occupation measure  of a contracting McKean-Vlasov SDEs or  that of its (simulable) time discretization schemes toward its invariant distribution (see e.g.~\cite{DuRenetal23},~\cite{ChassPag24}).  It involves to analyze the mean convergence rate for the underlying stationary Brownian diffusion connected with the MkV which exists as soon as the MkV SDE has a stationary regime.

%Then we can use, thanks to the splitting procedure developed in~\cite{Dereich13} and~\cite{FG}, the rate of convergence to $0$ of  $\|\nu_t(\omega,A)-\nu(A)\|_r$ (with $r=1$ or $2$ in practice) on hypercubes $A$ of $\R^d$ similar to the unit hypercube  to estimate $\big\| \mathcal{W}_p(\nu_t(\omega,dx,\nu)\big)\|_{L^p(\P)}$. 
%{We are going to show that under  a generic assumption of \textit{conditional} convergence to equilibrium} and sufficient control of the moments of the process, these rates can be}

 %we obtain the following quasi-generic estimate {for $p> 0$}:
%\[
%\big\| {\cal W}_p(\nu_t,\nu) \big\|_p\le C_{p,d,\varepsilon} t^{-\left(\frac{1}{2p}\wedge\left(\frac{1}{d}-\ve\right)\right)} \quad \mbox{ for every small enough $\ve>0$}.
%\]
%for some real constant $C_{p,d,\varepsilon}\!\in (0, +\infty)$. In other words

%Thus, in a Markovian setting, under appropriate mean-reverting/contraction and moment assumptions, we obtain the following quasi-generic estimate {for $p> 0$}:
%\[
%\big\| {\cal W}_p(\nu_t,\nu) \big\|_p\le C_{p,d,\varepsilon} t^{-\left(\frac{1}{2p}\wedge\left(\frac{1}{d}-\ve\right)\right)} \quad \mbox{ for every small enough $\ve>0$}.
%\]
%for some real constant $C_{p,d,\varepsilon}\!\in (0, +\infty)$. This rate of decay is faster than those obtained by a regularization approach. 

Before going further, let us note that in a companion paper (see~\cite{PagPanDiscrep}), we propose another field of application of this methodology to compare the rate of convergence with respect to Kolmogorov-Smirnov distances. When the distributions are $[0,1]^d$-supported (one often speaks in such a framework of \textit{discrepancy at the origin}) this yields  the rate of convergence in the Quasi-Monte Carlo Method for functions with finite variation (see~\cite{Nied1992} or~\cite{GilPag2026}).

%This idea is natural for basically two reasons: first the  $K$--$S$ distance relies on comparisons between distributions on hypercubes obtained by dilatation-translation form the unit hypercube, secondly  when dealing with $[0,1]^d$-supported distributions, in particular the uniform one, the $K$--$S$ is simply the so-called {\em discrepancy at the origin} which controls the convergence rate rate in the so-called Quasi-Monte Carlo Method for functions with finite variation in the Hardy \& Krause or measure sense (see~\cite{Nied1992} or~\cite{Pagesbook}).

The paper is organized as follows:  Section~\ref{sec:2} contains our main theoretical results and some  applications. It  is three-fold. First we propose some technical extensions of the main results from~\cite{FG} and to some extent from~\cite{Dereich13} on  upper-{bounding} the  mean $p$-Wasserstein distance  between an occupation measure $\nu_t$ and {a} distribution $\nu$. In particular we prove that we can ``decouple" the measure occupation $\nu_t$ and $\nu$  and the $L^1(\P)$  distributions used to bound the error between these distributions which allows to handle naturally non-stationary settings. We then state a continuous time  abstract general result, Theorem~\ref{thm:abstrait}, inspired by its counterpart in~\cite{FG} but in which stationarity  plays no role {\em a priori}. All the results obtained on classes of processes investigated in the paper finally refer to this theorem. 

In view of more specific applications, several general criterions that follow in this section are drawn from this first  more universal result. Among them, a result {on} stationary Markov processes satisfying a Poincar\'e inequality in \cref{prop:Poincare} (this result is the closest to those established in~\cite{FG}), a still rather general criterion, based on conditional distribution of a right continuous adapted process  (see \cref{thm:possiblynonmarkov}) shows how one can proceed to consider non-stationary processes possibly sharing no Markov property. Nevertheless we first apply it to Markov processes having an invariant distribution but starting from any regime {(see Propositions~\ref{prop:Markovsetting} and~\ref{prop:markovcontraction}  and \cref{cor:Contraction} for criterions)}, based on a convergence of the semi-group in total variation and for $\mathcal{W}_1$ distance (for the last two) respectively to the invariant distribution  respectively. 

In the last subsection we consider applications, first to Brownian diffusions (Theorem~\ref{thm:diffusionsconfluentes} and Corollary~\ref{cor:A tous moments}) or simply weakly mean-reverting (see~Theorem~\ref{prop:weakreverting}) and secondly classes of non-Markovian SDEs driven by fractional Brownian motions {fBm) or generalizations. This last part contains results for general Gaussian driving processes with some given memory properties (and appropriate Lyapunov conditions), then for general equations driven by fBm and finally in the specific case of Ornstein-Uhlenbeck SDEs where explicit computations on the covariance allow to improve the bounds on the conditional rate (see Theorem~\ref{cor:fractrao},  Corollary~\ref{cor2:fracappl} and Theorem~\ref{thm:specou})}. This non-Markovian application also provides a family of applications of our main result  where the rate of \textit{conditional convergence} to equilibrium may be non-integrable. 
  
  {Section~\ref{sec:proofmain} is devoted to the proof of Theorem~\ref{thm:abstrait} and \cref{sec:4bis} to the  proof of Theorem~\ref{thm:possiblynonmarkov} and the more tractable criterions drawn from it. Sections~\ref{sec:5bis} and \ref{sec:6} are respectively devoted to proofs  related to the applications to Brownian diffusions and fractionally driven SDEs.}
\normalsize
\paragraph{Notation}  
\noindent $\bullet$  {$\langle\cdot\,,\, \cdot\rangle$  and} $|\cdot|$  respectively denote the canonical inner product and Euclidean norm on $\R^r$, $r\in \N$.  

\noindent $\bullet$  Let $A\!\in \mathbb{M}_{d,d'}(\R)$ be a matrix with $d$ rows and $d'$ columns and real valued entries. Its Fr\"obenius norm, denoted $\|A\|_{_F}$, is defined by  $\|A\|_{_F}= \big( \sum_{ij}a_{ij}^2\big)^{\frac 12}$.

 \noindent $\bullet$    ${\cal P}_0(\R^d)$  denotes the convex set of probability distributions on $(\R^d, {\cal B}or(\R^d))$. It is usually equipped with the Baire $\sigma$-field $\sigma(\mu\mapsto \int fd\mu, \, f\!\in C_b(\R^d,\R))$ which turns out to be the Borel $\sigma$-field of the (metrizable) weak topology. 
 
\noindent  $\bullet$ For every $p>0$,  
%As the invariant distribution $\nu $ is unique, it follows form  {Let $p\ge 1$.} 
${\cal P}_p(\R^d)$  denotes the  convex subset of ${\cal P}(\R^d)$ of probability distributions  having a finite $p$-moment.

%\noindent $\bullet$   {@Fabien: J'ai enlev\'e la d\'ef de $\mathcal{W}_p$ ici car mise en d\'ebut d'intro. \`A voir}. 
%When $p\ge 1$,  $\W_p$ denotes the $p$-Wasserstein distance  on this set   defined for every $\mu$, $\mu'\!\in {\cal P}_p(\R^d)$ by
%\[
%  \W_p(\mu, \mu') = \inf\bigg\{\Big(\int |\xi-\xi'|^p \pi(d(\xi,\xi'))\Big)^{\frac 1p}, \pi(d\xi\times \R^d)= \mu,\; \pi(\R^d \times d\xi') =\mu'\bigg\}<+\infty.
%\]
%We know (see e.g.~\cite{Bolley08}) that $({\cal P}_p(\R^d), \W_p)$ is a Polish space (i.e. separable and complete) for every $p\ge1$.
%When $p\!\in (0,1)$,  $\W_p$ is still well-defined for distributions on  ${\cal P}_p(\R^d)$, however $\W_p$ is no longer a distance but $\W_p^p$ is, which justifies to provide results in that case.

\noindent $\bullet$   Let $q\!\in (0, +\infty)$. To alleviate notation, we will denote $\|\cdot\|_q:= \|\cdot\|_{L^q(\P)}$ the $L^q(\P) $-norm (or pseudo-norm) with respect to $\P$ defined on the   spaces $L_{\ER^r}^q(\Omega,{\cal A}, \P)$ of random variables or vectors $X:(\Omega,{\cal A})\to \ER^r$ such that $\E\, |X|^q<+\infty$.

\noindent $\bullet$   Let $q>0$. $ \mathbb{L}^q([0, +\infty))$  denotes the set of  $q$-integrable functions $f:[0, +\infty)\to \R$ with respect to the Lebesgue measure i.e. such that $\int_{[0, +\infty)} |f(y)|^qdu<+\infty$. 

\section{Main Results}\label{sec:2}
 \subsection{An abstract extension of Fournier--Guillin's Theorem}
 {In this section, we propose an extended version of the celebrated first Fournier--Guillin Theorem from~\cite{FG}. For that purpose, {we consider $(X_t)_{t\ge 0}$ an  $\R^d$-valued $(\F_t)_{t\ge 0}$-adapted right continuous  process -- hence ${\cal B}or(\R_+)\!\otimes \!{\cal A}$-bi-measurable --  defined  on a  stochastic basis $(\Omega, {\cal A}, (\F_t)_{t\ge 0}, \P)$ satisfying the usual conditions. We define for every $t>0$ and every $\omega\!\in \Omega$ the {\em occupation} measure by} 
\begin{equation}\label{eq:nuempiric}
\nu_t (\omega, d\xi)= \frac 1t \int_0^t \delta_{X_s(\omega)}(d\xi)ds\!\in {\cal P}(\R^d)
\end{equation}
which makes up a family of random distributions satisfying for every $p\ge0$, $t>0$ and  every $A\!\in  \B or(\R^d)$, 
\[
\Omega \ni \omega\mapsto \nu_t(\omega,A)\!\in {\cal P}_p(\R^d) \mbox{ is a $[0,1]$-valued}, {\cal F}_t\mbox{-measurable random variable on } (\Omega, {\cal A}).
\] 
Then, the random measure $\nu_t(\cdot,dx):\Omega\to {\cal P}_0(\R^d)$ is $\F_t$-measurable so that the family is $(\F_t)_{t>0}$-adapted. The same is true when viewing $\nu_t(\omega, d\xi)$ as a ${\cal P}_p(\R^d)$-valued random measure, $p>0$.

{We denote by $\nu\!\in {\cal P}_p(\R^d)$ a probability distribution which \textit{is an attractor} for $(\nu_t)_{t\ge0}$ in a sense made precise below (see \textit{e.g.} $\AUNgamma(i)$). The distribution $\nu$, e.g. in a Markovian framework but not only will appear to be a ($1$-marginal) invariant distribution of the process $(X_t)_{t\ge 0}$ in the sense that $X_t \sim \nu$ for every $t>0$ if $X_0 \sim \nu$.} 

{Now, let us recall} some notations from~\cite{FG} and a proposition which is the combination of Lemma{s}~5 and~6 from~\cite{FG} and  will be the key result on which we rely throughout this (part of) the paper. 

For every integer $\ell\ge 1$ we first introduce  the tiling (or partition) ${\cal P}_{\ell}$  of the hypercube $(-1,1]^d$ by $2^{d\ell}$ translations of the centered hypercube $(-2^{-\ell}, 2^{-\ell}]$.   We also introduce the sequence of $\ell^{\infty}$-rings $({\cal B}_n)_{n\ge 0}$ defined by
\[
\B_0= (-1,1]^d, \quad \B_n= (-2^{n}, 2^{n}]^d \setminus  (-2^{n-1}, 2^{n-1}]^d, \; n\ge 1.
\]
We also borrow from~\cite{FG} the notation $M_{\mu}(q)$ for the $\L^q(\mu)$-moment of a probability  distribution $\mu$   on $(\R^d, {\cal B}or(\R^d))$  w.r.t. to the norm $|\cdot|$, defined by
\[
M_{\mu}(q) = \int_{\R^d}|\xi|^q\mu(d\xi).
\]

Note that, for every $n\ge 1$,  
\begin{equation}\label{eq:q-Moment}
\mu(\B_n)\le \mu\big( \xi: |\xi|_{\infty}\ge 2^{n-1}\big)\le \int_{\R^d}|\xi|^q_\infty\mu(d\xi) 2^{-(n-1)q}\le K_d M_{\mu}(q)2^{-(n-1)q}.
\end{equation}
where $|\cdot|_{\infty}$ denotes the $\ell^{\infty}$-norm and $K_d=\sup_{|\xi|\le 1}|\xi|_\infty$.

The following Proposition, which is the combination of Lemma~5 and~6 from~\cite{FG}, is the key result on which we rely in this paper. 
\begin{prop}[A first general non-asymptotic bound] Let $p\!\in (0, +\infty)$ and let $d\ge1$.  There exists a positive constant $K_{p,d}$ such that for every  $\mu,\nu\!\in{\cal P}_p(\ER^d)$
\begin{equation}\label{eq:W-major}
{\cal W}_p^p(\mu,\nu) \le K_{p,d} \sum_{n\ge 0}2^{pn} \sum_{\ell\ge 0}2^{-p\ell}\sum_{F\in {\cal P}_{\ell}}\big| \mu(2^nF\cap \B_n)Ê- \nu(2^nF\cap \B_n)Ê\big|,
\end{equation}
where $2^nF = \{2^nx, \; x\!\in F\}$. 
\end{prop}

%\subsection{A general result in continuous time} We now state a natural extension to continuous time stochastic processes ofLet $(\Omega, {\cal A}, (\F_t)_{t\ge 0}, \P)$  denote a stochastic basis on which is defined  an homogeneous $\R^d$-valued $(\F_t)_{t\ge 0}$-adapted c\`ad process 
%Throughout the paper, we will denote by $\nu_t$ its occupation measure defined $a.s.$ and for all $t>0$ by:
%
%\begin{equation}\label{eq:nuempiric}
%\nu_t (\omega, d\xi)= \frac 1t \int_0^t \delta_{X_s(\omega)}(d\xi)ds. 
%\end{equation}
%

The following statement  is an extension of  a result originally established in~[Section 3,~\cite{FG}]. Here, we propose in a a continuous-time setting to provide a set of alleviated conditions (see~\cref{rem:replacebymu} for details) which allow to obtain a bound on the expected Wasserstein distance between the occupation measure and its target (invariant distribution in a Markov setting).

\begin{prop}\label{prop:fund} Let $\nu\!\in {\cal P}(\R^d)$ and let $\nu_t(\omega, dx))_{t>0}$ be a ${\cal P}(\R^d)$-valued family of random distributions  defined  on a probability space $(\Omega, {\cal A}, \P)$.
Let $p\!\in (0, +\infty)$  and let $q>p$. Let $\tzero >0$, $\beta,\gamma\!\in (0, \frac 12]$, $K_{\tzero,\b}\!\in (0,+\infty)$ and, for every $t\ge \tzero$,  let  {$\pi=(\proba_t)_{t\ge \tzero}$ be probability distributions such that the following assumption holds true}
% $\nu_t$ satisfies
\begin{equation}\label{eq:Abeta}
\AUNgamma\left\{\begin{array}{ll} (i)&\forall\, t\ge \tzero, \; \forall\, A \!\in {\cal B}or(\R^d),\; \E\, \big| \nu_t(A)-\nu(A)\big| \le K_{\tzero,\b}\min \bigg(\proba_t(A), \frac{\proba_t(A)^{\b}}{t^\gamma}\bigg)\\
(ii) & {\Mpiq}:= {1\vee}\sup_{t\ge t_0}  {M_{\proba_t}(q))}<+\infty.
\end{array}\right.
\end{equation}
%\tcr{F: le $1\vee$ me semble n\'ecessaire car il y a du $M(q)
%\tcr{\tcr
%\tcr{@G: j'ai ajout\'e $q$ et $\mathbf{P}$ en indice, pas certain que ça soit absolument n\'ecessaire de faire apparaître cette d\'ependance, je te laisse d\'ecider (voir par exemple prop. sur Poincar\'e pour te faire une id\'ee.}
Then,  there exists a real constant $K_{\b,p,d,\tzero,\gamma}$ such that,  for every $t\ge \tzero$, 
\begin{equation}\label{eq:UpperW_p}
\E\, {\cal W}_p^p\big(\nu_t, \nu) \le K_{\b,p,d,\tzero,\gamma}{2^q} {\Mpiq}\sum_{n\ge 0} 2^{pn} \sum_{\ell\ge 0} 2^{-p\ell} \min\bigg( 2^{-nq} ,\frac{2^{-nq\b}}{t^\gamma} 2^{d\ell(1-\b)}  \bigg).
\end{equation}
\end{prop}
%{@Gilles: j'ai enlevé la dépendance en $q$ mais ça fait apparaître un $2^q$.}
\begin{rem}\label{rem:replacebymu}
$\blacktriangleright$ As we will see later, providing a version of this result involving some general $\beta$ will allow to considerably extend the scope of application of the FG-theorem. As well, the reader can remark that the above result depends on the distribution $\proba_t$ only by its moments. By this relaxation, we will be able to consider dynamics in non-stationary regime.

\noindent $\blacktriangleright$ {The attractive properties of $\nu$ in $\AUNgamma(i)$ imply that $\nu$ is an invariant distribution in the Markovian setting.  However, $\nu$ must be more generally viewed as an attractor since the result may apply to non-Markovian settings such as asymptotically homogeneous dynamics, McKean-Vlasov SDEs of SDEs driven by fractional Brownian fBm}.
%Nevertheless, it can be noted that the above result does not use the invariance property and may be applied in non-Markovian cases, where $\nu$ can be viewed as an attractor (think for instance to asymptotically homogeneous dynamics, to McKean-Vlasov SDEs for instance or to SDEs driven by fBms).{Petite répétition avec ce qui précède.}
\end{rem}

%\tcr{@Gilles : If the RHS involves another distribution $\mu$, the result is still true.}
\noindent{\em Proof.}  It follows from $\AUNgamma$ that
\[
\hspace{-1cm}\sum_{F\in {\cal P}_{\ell}}\Exp\,\big| \nu_t(2^nF\cap \B_n)- \nu(2^nF\cap \B_n)\big|\le \min\bigg( \underbrace{2 \sum_{F\in {\cal P}_{\ell}} \pi_t(2^nF\cap \B_n)}_{= 2\,\pi_t(\B_n)}, \, t^{-\gamma}\sum_{F\in {\cal P}_{\ell}}  \proba_t(2^nF\cap \B_n)^{\b}\bigg).
\]

Let $r= \frac{1}{\b}$, $s= \frac{1}{1-\b}>1$ be two H\"older conjugate exponents. One has 
\begin{align*}
\sum_{F\in {\cal P}_{\ell}}  \proba_t(2^nF\cap \B_n)^{\b} & \le \Big(\sum_{F\in {\cal P}_{\ell}} 1\Big )^{1-\b} \Big( \sum_{F\in {\cal P}_{\ell}} \proba_t(2^nF\cap \B_n)\Big)^{\b}\\
& = 2^{d\ell(1-\b)} \proba_t(\B_n)^\b,
\end{align*}
so that
\[
\sum_{F\in {\cal P}_{\ell}}\Exp\,\big| \nu_t(2^nF\cap \B_n)Ê- \nu(2^nF\cap \B_n)Ê\big|\le \min\bigg(2\proba_t(\B_n), \frac{2^{d\ell(1-\b)} \proba_t(\B_n)^\b}{t^\gamma}\bigg).
\]
 By~\eqref{eq:q-Moment}, $\proba_t(\B_n)\le K_d{\Mpiq}2^{-q(n-1)}$, which yields the announced result.\hfill$\Box$

\bigskip
This yields the following abstract  result whose formulation slightly differs from that in~\cite{FG} since we emphasize the fact that $q$ is a structural variable linked to $\nu_t$ and the choice of $p<q$ is free.
%\tcr{Stabiliser partout la d\'ependance en $q$: met-on une constante qui d\'epend de tous les par. dont $q$...}
\begin{thm}\label{thm:abstrait}
%\dots (cf. proof ci-apr\`es)
Let $\nu$ and $\nu_t$, $t>0$, be as in the previous proposition. Let $p$, $q>0$, $q>p$, be such that $\nu(|\xi|^q)<+\infty$. Assume that $\AUNgamma$ is in force (or simply that~\eqref{eq:UpperW_p} holds true)  for some $\b\!\in (0, \frac 12]$ and some $\tzero>0$. Then there exists a real constant $K= K_{\b,p,{q},d,\tzero}>0$ such that for every $t\ge \tzero$,
\[
\E\, {\cal W}^p_p(\nu_t,\nu) \le K {{\Mpiq}}\left\{ \begin{array}{ll}
			 t^{- \frac{\gamma p}{d(1-\b)}} \mbox{\bf 1}_{\{\frac pq<\frac{d}{d+q}\}}+t^{- \frac{\gamma(q-p)}{q(1-\b)}}\mbox{\bf 1}_{\{\frac{d}{d+q} < \frac pq<1\}}&  \mbox{ if $p<d(1-\b)$},\\
 			% and $\b q<p<q$},\\
			%t^{-\frac 12}&  \mbox{ if $p>d(1-\b)$ and $p<\b q$},\\
			t^{- \gamma }\log(1+t)\mbox{\bf 1}_{\{ \frac pq<\b\}}+t^{- \frac{\gamma(q-p)}{q(1-\b)}}\mbox{\bf 1}_{\{\b < \frac pq<1\}}&  \mbox{ if $p=d(1-\b)$},\\
			% and  $\b q<p<q$},\\
			%t^{- \frac 12}\log(1+t)&  \mbox{ if $p=d(1-\b)$ and  $p<\b q$},\\
			t^{-\gamma}\mbox{\bf 1}_{\{ \frac pq<\b\}}+t^{-\frac{\gamma(q-p)}{q(1-\b)}}\mbox{\bf 1}_{\{\b < \frac pq<1\}}  & \mbox{ if $p>d(1-\b)$}.
			% and $\frac{dq}{d+q}<p<q$},\\
			%t^{- \frac{p}{2d(1-\b)}}&  \mbox{ if $p<d(1-\b)$ and $p<\frac{dq}{d+q}$}.\\
\end{array}\right.
\]
%\tcr{Que penses-tu d'\'ecrire plutôt
%\[
%\E\, {\cal W}_p(\nu_t,\nu)^p \le K \begin{cases} t^{-\frac 12} &  \mbox{ if $p>d(1-\b)$ and $p<\b q$},\\
%t^{- \frac 12}\log(1+t)&  \mbox{ if $p=d(1-\b)$ and  $p<\b q$}\\
%t^{-\Theta(p,q,d,\b)} &\mbox{else,}
%\end{cases}
%\]
%with $\Theta(p,q,d,\b)\in(0,1/2)$ and satisfies $\Theta(p,q,d,\b)=\ldots..$.
%}
%\tcb{@Fabien: pas trop d'accord car l\`a on est au coeur du truc quand m\^eme. En particulier quand $d$ est grand on est dans le dernier cas, celui qui lorsque $\beta = \frac 12$ donne $\| {\cal W}_p(\nu_t,\nu)\|_{L^p} = O(t^{-\frac 1d})$ qui est le r\'esultat que ``les gens'' retiennent. En tout  cas pour moi ce sont les cas $p<\ldots$ qui sont les plus important m\^eme si j'ai raisonn\'e  \`a $p$ d\'ecroissant. En pratique $p=1$ ou $2$ alors que  $d$ peut \^etre grand, surout par les temps qui courent \dots  F\& G adoptent une strat\'egie interm\'ediaire en mariant les 6 cas 2 par 2. Je voulais me d\'emarquer.}
In particular, if $\AUNgamma$ holds for some $q>\frac{p}{\beta}$, then there exists  a finite constant $K=K_{\b,p,d,\tzero}>0$ exists such that

\begin{equation}\label{eq:toutun}
\big\| {\cal W}_p(\nu_t,\nu) \big\|_p=\E[{\cal W}^p_p(\nu_t,\nu)]^{\frac{1}{p}}\le {K}[{{\Mpiq}}]^{\frac 1p}\left\{ \begin{array}{ll}
 t^{-\frac{\gamma}{d(1-\beta)}}&\mbox{ if $p< d(1-\beta)$}\\
 t^{-\frac{\gamma}{p}}(\log(1+t))^{-\frac 1p}&\mbox{ if $p=  d(1-\beta)$}\\

t^{-\frac{\gamma}{p}} &\mbox{ if $p>d(1-\beta)$}.
							
							\end{array}\right.
\end{equation}
%In the sequel,~\eqref{eq:toutun} will be shortly (and slightly abusively) written:
%$$\big\| {\cal W}_p(\nu_t,\nu) \big\|_p\lesssim t^{-\frac{1}{2p}\wedge\left(\frac{1}{d}\right)^{-}}.$$
\end{thm}

The proof of this theorem is postponed to Section~\ref{sec:proofmain}. {The constant $K$ may depend on $q$ but not on the distribution unlike ${{\Mpiq}}$. {This explains why we do not hide the constant ${{\Mpiq}}$ in  the constant $K$.} {Note that in this paper where our objective is to offer a very general framework in which these methods can be applied, we have chosen to leave aside the optimization of constants (this would clearly make the discussion too technical).} %\tcr{@Gil: revoir l'ordre ? Faire cro\^itre en $p$ ?}
%\noindent \textcolor{magenta}{@G \&F: 1/ Formuler ainsi comme Fournier-Guillin ou en $\| \, {\cal W}_p(\nu_t,\nu)\|_p$? 
%\\
%2/ L\`a je reproduis exactement la formulation de [FG] mais je pr\'e\`ererais utiliser celle que je prends plus tard dans les propositions o\`u j'encadre $p$ avec des trucs en $q$ et pas le contraire car $q$ est une donn\'ee du pb et $p$ un choix d'objet auquel on s'int\'eressel.

%We begin by recalling a (slightly extended) continous-time version of~\cite[Theorem 15]{FG}. 
In the sequel, $(P_t)_{t\ge0}$ denote{s} a Markov semi-group with invariant distribution $\nu$ and $ {\rm Var}_{\mu}(f)=\|f-\mu(f)\|_{\mathbb{L}^2(\mu)}$ and $$\nu_t(\omega)=\frac{1}{t}\int_0^t \delta_{X_s} ds.$$

\subsection{General criterions: from Poincar\'e to non-stationary non-Markov processes}\label{sec:genecrite}
In this section, we propose to provide general criterions which ensure {$\AUNgamma$} and in turn~\cref{thm:abstrait}.
\subsubsection{Poincar\'e setting} \label{subsec:Poincaree}
In~\cite{FG}, it was shown that such an assumption is fulfilled when Poincar\'e inequality holds. We begin by stating a slightly extended continuous-time version of~\cite[Theorem 15]{FG}.

%The setting is the closest to the discrete time mixing framework considered in~\cite{FG}. This is why we present it as a first illustration of Theorem~\ref{thm:abstrait}.

\begin{prop}[Under Poincar\'e-type Inequality] \label{prop:Poincare} Let $(P_t)_{t\ge 0}$ be a Markov transition semi-group  on $\R^d$ having  $\nu$   a unique stationary distribution $\nu$ such that  $M_\nu(q)=\int|x|^{q}\nu(dx)<+\infty$ for some  $q>1$. Assume that  there exists $\ee\!\in \mathbb{L}^1([0, +\infty))$ such that the following Poincar\'e-type inequalities hold
\begin{equation}\label{eq:VrarPoincare}
\ADEUX\hspace{1cm} \forall \,t\!\in  [0, +\infty),\;\forall\, f \!\in\mathbb{L}^2(\nu),\quad {\rm Var}_{\nu}(P_t f)\le \ee(t) {\rm Var}_{\nu}(f).
 \end{equation}
When $X_0\sim{\nu}$, \emph{i.e.} when  $(X_t)_{t\ge0}$ is under its stationary regime, for every $\tzero>0$, Assumption   $(\mathbf{A}_{\frac{1}{2},q,\frac{1}{2}})$ holds true with $\proba_t=\nu$,  ${\Mpiq}=M_{\nu}(q)$.
Then there exists a real constant $K= K_{\b,p,q,d,\tzero}>0$ such that for every $t\ge \tzero$,
\[
\E\, {\cal W}^p_p(\nu_t,\nu) \le K\left\{ \begin{array}{ll}
			 t^{- \frac{p}{d}} \mbox{\bf 1}_{\{\frac pq<\frac{d}{d+q}\}}+t^{-(1- \frac{p}{q})}\mbox{\bf 1}_{\{\frac{d}{d+q} < \frac pq<1\}}&  \mbox{ if $p<\frac d2$},\\
			t^{- \frac 12}\log(1+t)\mbox{\bf 1}_{\{ \frac pq<\frac 12\}}+t^{-(1- \frac{p}{q})}\mbox{\bf 1}_{\{\frac 12 < \frac pq<1\}}&  \mbox{ if $p=\frac d2$},\\
 			t^{-\frac 12}\mbox{\bf 1}_{\{ \frac pq<\frac 12\}}+t^{-(1- \frac{p}{q})}\mbox{\bf 1}_{\{\frac 1 2< \frac pq<1\}}  & \mbox{ if $p>\frac d2$}.\\
			% and $\b q<p<q$},\\
			%t^{-\frac 12}&  \mbox{ if $p>d(1-\b)$ and $p<\b q$},\\
			
			% and  $\b q<p<q$},\\
			%t^{- \frac 12}\log(1+t)&  \mbox{ if $p=d(1-\b)$ and  $p<\b q$},\\
			
			% and $\frac{dq}{d+q}<p<q$},\\
			%t^{- \frac{p}{2d(1-\b)}}&  \mbox{ if $p<d(1-\b)$ and $p<\frac{dq}{d+q}$}.\\
\end{array}\right.
\]
\end{prop}
\noindent {The proof of this proposition is postponed to Section \ref{sec:Poincarre}.}

\medskip
\begin{rem} $\blacktriangleright$ As already mentioned in~\cite{FG} the  critical sub-cases  -- $\frac q2$ in the  two last cases and $\frac{dq}{d+q}$ in the first one --  could be treated separately and would introduce an additional $\log$ term. This would add still more technicalities for a small benefit.

  \noindent  $\blacktriangleright$ The terminology ``Poincar\'e-type inequality''  certainly follows from the fact that $\ADEUX$ holds if Poincar\'e's inequality holds. Actually (see \emph{e.g.}~\cite[Theorem 4.2.5]{Bakry-Gentil-Ledoux} for background), it is well known that a Poincar\'e inequality  with constant $C$ for $\nu$ is equivalent to $\ADEUX$ with $\ee(t)= e^{-\frac{2 t}{C}}$. 

  \noindent  $\blacktriangleright$ {To ensure $\AUNgamma$ with $\beta=\gamma=1/2$}, it is enough that the inequality of $\ADEUX$ holds for $f=1_{A}$ with $A\in{\cal B}(\ER^d)$. Nevertheless, by a density argument it can be checked that if $\ADEUX$ only holds for every indicator functions $1_A$ of Borel sets, then it also holds for every $f\in \mathbb{L}^2(\nu)$ (see~\cite[p.136]{Bakry-Gentil-Ledoux}).
  %tcr{En gros dans , ils expliquent qu'il suffit que l'inegalit\'e de Poincar\'e (celle sans decroissance expo) soit vraie sur une alg\`ebre ${\cal A}_0$ dense dans ${\mathbb{L}}^2(\nu)$, voir p.136\ldots)}

  \noindent $\blacktriangleright$ Assumption $\ADEUX$ being true as soon as $(\ee(t))_{t\ge0}$ is integrable on $[0, +\infty)$, it may extend to  more general settings than those related to classical Poincar\'e inequality. Nevertheless, it requires that the right-hand side is ``proportional'' to ${\rm Var}_{\nu}(f)$, which may be difficult to check in practice. 
  \end{rem}

\noindent \textbf{Example.}  It is well-known that the above  Poincar\'e inequalities $\ADEUX$  hold ({with $\ee(t)=e^{-\rho t}$, $\rho>0$}) for  $\R^d$-valued  Langevin diffusions reading
\[
dX_t = -\nabla U(x_t)dt + \sqrt{2}\sigma d W_t
\]
where $\s>0$ and $U\!\in {\cal C}^2(\R^d, \R)$ is a coercive convex function (see \textit{e.g.}~\cite{bobkovledoux97}). By a perturbation argument (see~\cite[Proposition 4.2.7]{Bakry-Gentil-Ledoux}), it is even true if $U$ is only convex outside a ball $B(0,M)$ for a given $M>0$  (as a compact perturbation of a convex function).
%(see also~\cite[Remark 3.2]{cattiauxguillin2014}). 
  In this case,  $\ADEUX$ holds true.

\subsubsection{A general criterion} \label{subsec:2.2.2}
From now on,  $\nu_t$ will denote the occupation  measure at time $t>0$ associated to a right continuous process $(X_t)_{t\ge 0}$ as defined in~\eqref{eq:nuempiric} and we denote by $\bar \nu_t$ the \textit{mean occupation measure} of $\nu_t$, defined for every $t>0$ by
\begin{equation}\label{eq:mut}
\bar \nu_t(f)=\E \,\nu_t(f) =\frac{1}{t}\int_0^t \E f(X_s) ds.
\end{equation}

%\tcr{A revoir avec le par. $\gamma$} 
We establish a general result that covers the Markov setting and beyond to  in which Assumption $\AUNgamma$ is satisfied with  various values of ${\gamma},\beta\!\in (0,\frac 12]$. Then we derive several criterions which are easier to fulfill in practice for continuous time Markov processes. Basically we start from assumptions of $TV$-convergence of conditional distributions of the process $(X_t)_{t\ge 0}$   toward the distribution $\nu$ under consideration. In a  homogeneous Markov setting this can be read on its  transition semi-group. As we want to establish some convergence rates in Wasserstein distance we also provide criterions based on Wasserstein distance. We also consider situations where the process $(X_t)_{t\ge 0}$ is not Markovian. 

% abstract focus on the case where a convergence in total variation holds towards $\nu$. In view of applications to fractional driven diffusions (or to other non-Markovian settings such as asymptotically homogeneous Markov processes for instance), we provide here a result which does not require the Markov property. 
 %and we adopt the convention $q/(q-1)=+\infty$ when $q=1$.
%do not necessarily assume that 
%$(X_t)_{t\ge0}$ is a Markov process. D
%In the proposition, we again assume that $(P_t)_{t\ge0}$ is a semi-group ({pas forc\'ement Feller ??}).
%{Proposition g\'en\'eralisation permettant de couvrir toutes les d\'ecroissances. Finalement, l'application au fBm se fait avec ce critère.}
\begin{thm}\label{thm:possiblynonmarkov} Let $(X_t)_{t\ge0}$ denote an $({\cal F}_t)_{t\ge0}$-adapted right continuous  process with values in $\ER^d$. Let $q>1$, $\gamma\in(0,1/2]$, {$\nu\in{\cal P}(\ER^d)$},  and a bounded function $e:\ER_+\rightarrow\ER_+$ such that
%such that some $c>0$ and  
%$e\!\in \mathbb{L}^1([0, +\infty),du)$ such that
\begin{equation}\label{eq:ATROIS}
\ATROIS\quad\begin{cases} (i)\;  \|{\cal L}(X_{t}|{\cal F}_s)-\nu\|_{TV}\le \Upsilon_{t,s} \,\ee(t-s)\;\mbox{  for all $s$, $t\!\in \R_+,\;0\le s\le t$},\\
(ii)\; \int_0^t e(s)ds \le O( t^{1-2\gamma})<+\infty \mbox{ as }t\to+\infty,  \\
(iii)\; C_q=\sup_{t\ge 0}\big(\sup_{s{\le}  t} \E\, \Upsilon_{t,s}^{q}+{\sup_{t\ge0}\E\,|X_t|^{q}}\big)<+\infty,
\end{cases}
\end{equation}
where, for  every $s\ge 0$, $\big(\Upsilon_{t,s}(\omega))_{ \omegaÊ\in \Omega, t\ge s}$ denotes a  family of  ${\cal B}or\big([s,+\infty)\big)\otimes{\cal F}_s$-measurable   non-negative random variables. 

Then $\nu(|\cdot|^q)<+\infty$  and
%
% with $\Upsilon_{t,s}$ ${\cal F}_t$-measurable $\Upsilon_{s,t}$ is a non-negative random variable satisfying
%$$\sup_{(s,t)\in[0, +\infty)^2, 0\le s\le t} \E \Upsilon_{s,t}^{\frac{\rho}{\rho-1}}=C<+\infty.$$
$\AUNgamma$ holds  for every $\tzero>0$ with $\beta=\beta(q)=\frac 12 (1-\frac{1}{q})$, $\pi_t=\frac{1}{2}(\bar \nu_t+\nu)$. Thus, 
for every $\tzero>0$,  there exists a positive constant $K= K_{b,\s,p,d,t_0}$ such that, for every $t\ge \tzero$, 
%(\tcb{@Gil: temporaire a inclure dans la preuve... the assumptions of Theorem~\ref{thm:abstrait} are satisfied with  $\pi_t= \frac 12 (\mu +\nu)$, 
%$\b= \frac 12(1-\frac 1q)$ and aby $\tzero>0$.
%{R\'eintroduire le $\tzero$ et l'in\'egalit\'e de BEL!!! \dots et son hypoth\`ses d'elliptict\'e} ... 
%pour la condition $P_t(f)$ Lipschitz!!!!}. For every $t_0>0$,  there exist a constant $K=K_{p,q,d,q,t_0}>0$ such that, for every $t\ge t_0$,

%{Version avec $\gamma$ (ancienne version en commentaire}
\begin{equation}\label{eq:boundsabstrait}
\E\, {\cal W}^p_p(\nu_t,\nu) \le K {\Mpiq}\left\{ \begin{array}{ll}
		t^{-\frac{{2\gamma} pq}{(q+1)d}}\mbox{\bf 1}_{\{p< \frac{dq}{d+q}\}}+t^{-\frac{{2\gamma}(q-p)}{q+1}}\mbox{\bf 1}_{\{  \frac{dq}{q+1}<p<q\}} &\mbox{if  $\,p< \frac{d}{2}(1+\frac 1q)$},\\
		t^{-\frac{{2\gamma}(q-p)}{q+1}} \mbox{\bf 1}_{\{d_-<q<d_+\}} + t^{-{\gamma}}\log(1+t) \mbox{\bf 1}_{\{d_+<q\}}& \mbox{if  $\,p= \frac{d}{2}(1+\frac 1q)$},\\
		 t^{-{\gamma}}\mbox{\bf 1}_{\{\frac{d}{2}\frac{q+1}{q}<p<\frac{q-1}{2}\}}+t^{-\frac{{2\gamma}(q-p)}{q+1}}\mbox{\bf 1}_{\{\frac{q-1}{2}<p<q\}}	&\mbox{if  $\,p>\frac{d}{2}(1+\frac 1q)$,}
		% and $\frac{q-1}{2}<p<q$ (with $q>\frac{d+1+\sqrt{(d+1)^2+4d}}{2}$)},\\
 		%t^{-\frac12}	   &\mbox{if  $\,p>\frac{d}{2}\frac{q+1}{q}$ and $0<p<\frac{q-1}{2}$},\\	
		
		%    and  $\frac{d+\sqrt{d^2+8d}}{2}<q<\frac{d+1+\sqrt{(d+1)^2+4d}}{2}$},\\	
		%t^{-\frac12}\log(1+t)	&\mbox{if  $\,p= \frac{d}{2}\frac{q+1}{q}$ and $\frac{d+\sqrt{d^2+8d}}{2}<q<\frac{d+1+\sqrt{(d+1)^2+4d}}{2}$},\\	
		
		% and $\,\frac{dq}{d+q}<p<q$},\\
		%t^{-\frac{pq}{(q+1)d}}&  \mbox{if $\,p<\frac{d}{2}\frac{q+1}{q}$ and $p< \frac{dq}{d+q}$}.			 	 
			\end{array}\right.
\end{equation}
where $d_-:= \frac{d+\sqrt{d^2+8d}}{4}<d_+ := \frac{d+1+\sqrt{(d+1)^2+4d}}{2}$ and  ${\Mpiq}$ satisfies ${\Mpiq} \le 1\vee \sup_{t\ge t_0}\E\,|X_t|^{q}$. 
%\begin{equation}\label{eq:boundsabstrait}
%\E\, {\cal W}^p_p(\nu_t,\nu) \le K M(q)\left\{ \begin{array}{ll}
%		t^{-\frac{pq}{(q+1)d}}\mbox{\bf 1}_{\{p< \frac{dq}{d+q}\}}+t^{-\frac{q-p}{q+1}}\mbox{\bf 1}_{\{  \frac{dq}{q+1}<p<q\}} &\mbox{if  $\,p< \frac{d}{2}(1+\frac 1q)$},\\
%		t^{-\frac{q-p}{q+1}} \mbox{\bf 1}_{\{d_-<q<d_+\}} + t^{-\frac12}\log(1+t) \mbox{\bf 1}_{\{d_+<q\}}& \mbox{if  $\,p= \frac{d}{2}(1+\frac 1q)$},\\
%		 t^{-\frac12}\mbox{\bf 1}_{\{\frac{d}{2}\frac{q+1}{q}<p<\frac{q-1}{2}\}}+t^{-\frac{q-p}{q+1}}\mbox{\bf 1}_{\{\frac{q-1}{2}<p<q\}}	&\mbox{if  $\,p>\frac{d}{2}(1+\frac 1q)$,}
%		% and $\frac{q-1}{2}<p<q$ (with $q>\frac{d+1+\sqrt{(d+1)^2+4d}}{2}$)},\\
% 		%t^{-\frac12}	   &\mbox{if  $\,p>\frac{d}{2}\frac{q+1}{q}$ and $0<p<\frac{q-1}{2}$},\\	
%		
%		%    and  $\frac{d+\sqrt{d^2+8d}}{2}<q<\frac{d+1+\sqrt{(d+1)^2+4d}}{2}$},\\	
%		%t^{-\frac12}\log(1+t)	&\mbox{if  $\,p= \frac{d}{2}\frac{q+1}{q}$ and $\frac{d+\sqrt{d^2+8d}}{2}<q<\frac{d+1+\sqrt{(d+1)^2+4d}}{2}$},\\	
%		
%		% and $\,\frac{dq}{d+q}<p<q$},\\
%		%t^{-\frac{pq}{(q+1)d}}&  \mbox{if $\,p<\frac{d}{2}\frac{q+1}{q}$ and $p< \frac{dq}{d+q}$}.			 	 
%			\end{array}\right.
%\end{equation}
%where $d_-:= \frac{d+\sqrt{d^2+8d}}{4}<d_+ := \frac{d+1+\sqrt{(d+1)^2+4d}}{2}$ and  $M(q)$ satisfies $M(q) \le \sup_{t\ge t_0}\E\,|X_t|^{q}$.

 In particular, if $C_q<+\infty$ for any $q>0$, then for every $d\in\mathbb{N}^*$, $p>0$ and $\varepsilon>0$, a finite positive constant $\bar{K}= \bar{K}_{p,\ve}$ exists such that
\begin{equation}\label{eq:toutunbis}
\big\| {\cal W}_p(\nu_t,\nu) \big\|_p\le \bar{K}     
							\left\{ \begin{array}{ll}
 t^{-\frac{{2\gamma}}{d}(1-\ve)}&\mbox{ if $p\le \frac d2$}\\
t^{-\frac{{\gamma}}{p}} &\mbox{ if $p>\frac d2$}.
							\end{array}\right.
\end{equation}
In the sequel,~\eqref{eq:toutunbis} will be shortly (and slightly abusively) written:
$$
\big\| {\cal W}_p(\nu_t,\nu) \big\|_p\lesssim t^{-{2\gamma}\left(\frac{1}{2p}\wedge\left(\frac{1}{d}\right)^{-}\right)}.
$$
% and some (finite) constants $K$ and $M(q)$ depending (only) on  $q$, $C$ and $\|e\|_{\mathbb{L}^1}$.
% ({voir pour valeur de $t_0$}).
\end{thm}

 The proof of this theorem is postponed to~\cref{subsec:proofthm2.5}}.
%{pour \bar{K}, je ne mets pas les d\'ependances car si je mets $\varepsilon$ en +, ça veut dire que 
\smallskip
\begin{rem}\label{rq:allevatrois} $\blacktriangleright$ Following carefully the proof, one checks  that     $\ATROIS(ii)$-$(iii)$ can be replaced by:
\begin{equation}\label{eq:allevatrois}
\max\Big(\underbrace{\sup_{t\ge 0}t^{2\g -1}\sup_{0\le u\le t}\int_0^{t} { \| \Upsilon_{(u+s)\wedge t,u}\|_q }\ee(s) ds}_{C_{1,q}},\underbrace{\sup_{t>0}\frac{1}{t}\int_0^t \E |X_s|^{q} ds}_{C_{2,q}}\Big)<+\infty.
\end{equation}
{This remark will be useful for \cref{prop:weakreverting} where one considers weak mean-reverting drifts. Actually, in these cases, one may be able to prove that $C_{2,q}<+\infty$ but not that $C_q<+\infty$ (precisely when the parameter $a$ of this theorem is strictly lower than $1$).}
%\tcb{@Fabien:  New assumption qui me semble plus naturelle dans un cadre de $L^q$-bornitude. Ceci dit comme,  le plus vraissemblable est que c'est born\'e (car TV) et au pire \`a croissance lin\'eaire  dans la cas contractif}. \tcb{J'ai chang\'e $C_{1,q}$ en 
%\[
%C_{1,q} :=\sup_{t\ge 0}\int_0^{+\infty}  \|\Upsilon_{t+s,t}\|_{q}\ee(s) ds <+\infty
%\]
%pour avoir $\beta = \frac 12(1-\frac 1q)$ qui tend vers $\frac 12$ quand $q\to+\infty$.} 
%J'ai supprim\'e $t_0$ de l'\'enonc\'e $t\ge 0$ ou $t\ge t_0$. ici $t_0$ n' a gu\`ere de sens et comme $(X_t)$ est au moins c\`ad  je pense 
%que  $\frac{1}{t}\int_0^t \E |X_s|^{q}ds$ reste born\'e au voisinage de de $0$ car tend  vers  $\E |X_0|^q$.}

\noindent $\blacktriangleright$ The constant $K_{t_0,\b}$ involved in $\AUNgamma(i) $ is  $K_{t_0,\b}=2(1\vee \sqrt{C_{1,q}})$ and one has ${\Mpiq}\le  C_{2,q}\le C<+\infty$ (see~\eqref{eq:constanteexplca2} in the proof and the lines that follow).

\noindent  $\blacktriangleright$ Note that $q>d_-$ is a necessary condition on $q$ to satisfy $q> \frac{d}{2}(1+\frac 1q)$ i.e. to ensure that the above last two cases in~\eqref{eq:boundsabstrait} are not   empty. Note that when $d=1$, $d_-=1$ and $d_+= 1+\sqrt{2}$.

\noindent $\blacktriangleright$ {In most situations, we will consider situations where $\gamma=1/2$, \emph{i.e.} where $\ee\!\in \mathbb{L}^1([0, +\infty))$. The case $\gamma<1/2$ will be useful for fractional SDEs.}
%{@ F\& G: checker car j'ai chang\'e des trucs}.
\end{rem}
\subsubsection{Criterions for (possibly non stationary) Markov processes}
In this section, we state consequences of~\cref{thm:possiblynonmarkov}  in the setting of  ergodic (but not necessary stationary) homogeneous Markov processes without the use of Poincar\'e inequality. {The proofs of this section are postponed to \cref{sec:proof4.2}}.

\noindent \textbf{$TV$ criterion.} We first rephrase $\ATROIS$ in this setting {with $\gamma=1/2$}. This yields:
\begin{prop}[Markov setting] \label{prop:Markovsetting}
Let $(P_t)_{t\ge0}$ denote a  Markov transition semi-group defined on bounded $\R$-valued Borel functions defined on $\R^d$. Let $\mu_0\!\in {\cal P}_q(\R^d)$ for some $q> 1$ 
and 
% which admits a unique invariant distribution $\nu$.  
%Assume that $X_0\sim \mu_0$ (such that $\E |X_0|^q<+\infty$). 
let $\psi:\ER\rightarrow[0, +\infty)$ and $\ee\!\in \mathbb{L}^1([0, +\infty))$  be  Borel functions such that,
\begin{equation}\label{eq:TVboundby}
\begin{cases}
(i)& \;\forall t\ge0,\forall\, x \!\in \R^d, \quad \|P_t(x,dy)-\nu\|_{TV}\le \psi(x)\ee(t),\\
(ii)& \;\sup_{t\ge 0} \mu_0P_t \left(\psi^{q}+ |\cdot|^{q}\right)<+\infty.
\end{cases}
\end{equation}
Then, $\nu$ is the unique invariant distribution  of  $(P_t)_{t\ge 0}$, $\nu\!\in {\cal P}_q(\R^d)$  and  $\AUN$ holds true for every $t_0>0$  with $\beta=\frac 12 (1-\frac{1}{q})$, $\pi_t=\frac{1}{2}(\bar \nu_t+\nu)$ so that the bounds~\eqref{eq:boundsabstrait} of~\cref{thm:possiblynonmarkov} are in force.
%.so that all conclusions of~\cref{thm:possiblynonmarkov}   are satisfied by any homogeneous Markov process $(X_t)_{t\ge 0}$ with transition $(P_t)_{t\ge 0}$ such that $X_0\stackrel{d}{\sim} \mu_0$.
\end{prop}
\begin{rem}\label{eq:allevatroisbis} $\blacktriangleright$ If $\mu_0=\nu$, \emph{(ii)} reads:
$$
\int_{\R^d} \left(\psi(x)^{q}+|x|^q\right) \nu(dx)<+\infty.
$$
 so that if $\psi(x) \le C(1+|x|)$, this condition simply reads $\int  |x|^q  \nu(dx)<+\infty$.

\noindent $\blacktriangleright$ As for Condition~\eqref{eq:allevatrois},~\eqref{eq:ATROIS}\emph{(ii)} can be replaced by the weaker moment assumption:
\begin{equation}\label{eq:condaltmark}
\left[\int_0^{+\infty}  (\mu_0P_t \psi^{q})^{\frac 1q} \ee(t) dt+\sup_{t> 0} \frac{1}{t}\int_0^t  \mu_0P_s |\cdot|^{q} ds\right]<+\infty.
\end{equation}
\end{rem}

\noindent \textbf{Wasserstein criterions.} It is natural to state a criterion based on Wasserstein, here Monge-Kantorovich, distance since our objective is to establish some convergence rate  based on Wasserstein distances.

%\textcolor{red}{Attention : voi r qd m\^eme o\`u on pourrait avoir besoindd e Feller !}
\begin{prop}\label{prop:markovcontraction} Let $(P_t)_{t\ge0}$ be  a  (strongly Feller)\footnote{{The strong Feller property is ensured by $(i)$ (which implies that $P_{\theta_0}g$ is continuous when $g$ is bounded measurable).}} Markov transition semi-group defined on bounded and nonnegative Borel functions on $\R^d$. Assume it admits a unique invariant distribution $\nu$. {Assume} 

$(i)$ $\LSF$ (for Lipschitz Strong Feller):  there exists $\theta_0>0$, $c(\theta_0)>0$   such that
\begin{equation}\label{eq:LipStrgFeller}
\hskip-1,25cm \; \mbox{ for every {bounded} Borel function $g: \R^d \to \R$, $P_{\theta_0}g$ is Lipschitz with $[P_{\theta_0}g]_{\rm Lip}\le c(\theta_0)\|g\|_{\sup}$}, 
\end{equation}

$(ii)$ {$\ATROISbis$} :  there exists $\mu_0\!\in {\cal P}_q(\R^d)$ for some $q> 1$ and $\psi:\ER\rightarrow[0, +\infty) $ and $\ee\!\in \mathbb{L}^1([0, +\infty))$  some Borel  functions such that 
\begin{equation}\label{eq:W1boundby}
\qquad \begin{cases}
(i) & \forall t\ge0,\forall\, x \!\in \R^d, \quad \W_1(P_t(x,dy),\nu)\le \psi(x)\ee(t),\\
 (ii) & \sup_{t\ge 0}\mu_0P_t(\psi^q+ |\cdot|^{q})<+\infty.
\end{cases}
\end{equation}
Then the assumptions of Proposition~\ref{prop:Markovsetting} are satisfied with $\psi\vee 1$ and $t \mapsto 2\mbox{\bf 1}_{\{t\le \theta_0\}}+ c(\theta_0)e(t-\theta_0)\mbox{\bf 1}_{\{t\ge \theta_0\}}\!\in {\mathbb L}^1(du)$ instead of $\psi$ and $e$ respectively so that $\nu(|\cdot|^q)<+\infty$ and all the other conclusions of Theorem~\ref{thm:possiblynonmarkov} are satisfied  for a Markov process $(X_t)_{t\ge 0}$ with transition $(P_t)_{t\ge 0}$ such that $X_0\stackrel{d}{\sim} \mu_0$. 
%Then, $\AUN$ holds for  $t\ge 2\theta_0$, $\beta=\frac{1}{2}(1-\frac 1q)$, $\pi_t=\widetilde \nu_t$ and some (finite) constants $K$ and $M(q)$ depending (only) on $\rho$, $e$ and $C$ (\tcr{voir pour la formulation de $\theta_0$}).
%
\end{prop}

 An important criterion for~\eqref{eq:W1boundby}$(i)$ to be satisfied is the so-called contraction framework as emphasized in the next corollary.
 
\begin{cor}[$L^1$-Contraction]  \label{cor:Contraction} Let $(P_t)_{t\ge0}$ denote a  {{Markov}} transition   semi-group  {with unique invariant distribution  $\nu$}. If this semi-group satisfies  the following contraction inequality:
\begin{equation}\label{eq:W1boundbyter}
\hskip-0.5cm \ATROISter \left\{\begin{array}{ll}
(i)& \; \hbox{there exists $\Psi:\R^2\to \R_+$, Borel function, and $\ee \!\in \mathbb{L}^1([0, +\infty))$ such that  }\\
&\;\forall t\ge0,\forall\, x,\, x'\!\in \R^d, \; \W_1(P_t(x,dy),P_t(x',dy))\le \Psi(x,x')\ee(t),\\
(ii) &\;\displaystyle \hbox{the functions }\psi  :=\int_{\R^d} \Psi(\cdot,x')\nu(dx')<+\infty \mbox{ and $|\cdot |$ satisfy } \ATROISbis(ii)\\
&\; \hbox{ for some $q> 1$ and some probability measure $\mu_0\!\in {\cal P}_q(\R^d)$, $q>1$.}
\end{array}\right.
\end{equation}

Then  $\nu(|\cdot|^q)<+\infty$ and the above assumption $\ATROISbis$ is satisfied.
\end{cor}

\begin{rem}   If $\psi(x) $ is well-defined  for every $x\!\in \R^d $, a sufficient condition for $\ATROISter(ii)$  that can be directly read on $\Psi$ is 
\[
 \sup_{t\ge 0} \int_{(\R^d)^2}(\Psi(\xi,\xi')^q+|\xi'|^q)\mu_0P_t(d\xi)\nu(\d\xi')<+\infty.
 \]
 This is an easy consequence of Jensen's inequality (which also shows that $\psi(x)$ is defined $\mu_0(dx)$-$a.s.$).
\end{rem}

%\noindent $\bullet$ \tcr{@F\&G: Sous cette forme on est plus ou moins oblig\'e de supposer soit que $\psi(x) = O(|x|)$, soit que $\nu$ int\`egre a priori $\psi$. Ainsin si l'on prend $\psi(x)= O(|x|^r)$ pour un $r>1$, soit  if faut prendre\dots $\beta= \frac 12(1-\frac rq)$ comme j'avais fait...ou cela revient \`a supposer a priori (de fa\c con cach\'ee) que $\nu(|\cdot|^{rq})<+\infty$.}

In the next  two sections, we apply the above criterions to two settings (among others): Brownian diffusions and fractional SDEs (which are not Markovian).
\subsection{Applications}
\subsubsection{Applications to Brownian diffusions}
{In this section, we apply our criterions to ergodic Brownian diffusions. An example has already been given in~\cref{subsec:Poincaree}. Here, we mainly focus on examples without resorting to Poincar\'e inequality.}  We consider the $\R^d$-valued  stochastic differential equation (SDE)
\begin{equation}\label{eq:EDSbis}
X_t =X_0+\int_0^t b(X_s)ds + \int_0^t \sigma(X_s)dW_s,
\end{equation}
where $W$ is a $d'$-dimensional standard Brownian motion, $X_0$ is a random vector defined on the same probability space,  independent of $W$, $b:\R^d\to \R$ and $\sigma:\R^d \to {\mathbb M}_{d,d'}(\R)$ are locally Lipschitz continuous
 and satisfy
\begin{equation}\label{eq:onR+}
\exists\, c>0, \; \forall\, x\!\in \R^d, \quad \tfrac 12\|\s(y)\|^2_F(x)+\langle b(y), y\rangle\le C(1+|y|^2).
\end{equation}
Then (see e.g.~\cite{Wang2020}) the SDE~\eqref{eq:EDSbis} has a unique strong solution on the whole nonnegative real line\footnote{Without this assumption the strong solution exists only until  an exploding time $\tau$ possibly infinite. Under condition~\eqref{eq:onR+} one easily shows applying It\^o's formula to $y\mapsto |y|^2$ that   $\E\, |X^x_{t\wedge \tau}t|^2\le (|x|^2+2C)e^{2Ct}$ so that \dots $\tau =+\infty$ $\P$-$a.s.$}. The associated  {Markov semi-group {$(P_t)_{t\ge0}$} is Feller} and the infinitesimal generator of the diffusion 
reads on $C^2(\R^d,\R)$ functions $g:\R^d\to \R$
\[
\L g = \langle b\,|\,g\rangle + \frac 12 {\rm Tr}(\s^{\top} \nabla^2g\, \s).
\]
{The rest of this section follows the same structure as the previous one, distinguishing between the total variation and Wasserstein approaches.We end this section by a {general} result which is less constraining in terms of  {the mean-reverting condition on $b$},  based on (usually non-quantitative) convergence in total variation to equilibrium.}

\noindent {\textbf{Criterions based on TV-convergence.}  The following result is based  on the TV-convergence to equilibrium obtained in \cite{doucfortguillin} which has the advantage to require very weak conditions of mean-reverting:}
\begin{thm}\label{prop:weakreverting} Assume that $d=d'$, that $b$ and $\sigma$ are locally Lipschitz functions, that $\sigma$ is bounded and that $\sigma$ is uniformly elliptic\footnote{In fact, the ellipticity assumption must be uniform only on compact sets.}:
\begin{equation}\label{eq:sigmaUellip}
\exists\,\ve_0>0,\,\forall x\in\ER^d,\quad  \s\s^{\top}(x)\ge \ve_0I_d\quad\textnormal{(in the sense of symmetric matrices)}.
\end{equation}
 Assume that $\mu_0$ has {finite} moments of any orders. Then, if there exist some positive $M$, {$a>0$ and $\alpha>0$} such that for $|x|>M$, $\langle b(x),x\rangle\le  -\alpha |x|^{2a}$, 
 %Let $V:\ER^d\mapsto \ER_+*$ denote a ${\cal C}^2$ function such that some $\rho>0$, $\beta\in\ER$, $\alpha>0$ and $a\in(0,1]$ exist such that
%$$\liminf_{|x|\rightarrow+\infty} \frac{V(x)}{|x|^\rho}>0  \quad\textnormal{and}\quad {\cal L}V\le \beta-\alpha V^a.$$ Assume the   uniform  ellipticity assumption~\eqref{eq:sigmaUellip}. 
$$\big\| {\cal W}_p(\nu_t,\nu) \big\|_p\lesssim t^{-\frac{1}{2p}\wedge\left(\frac{1}{d}\right)^{-}}\quad \textnormal{in the sense of~\eqref{eq:toutunbis}}.$$
\end{thm}
\noindent {The proof of this theorem is  postponed to~\cref{subsec:proof2.12}.} 
%{Even if the above result is very general, it is based on a highly none quantitative result. Actually, the strategy of \cite{doucfortguillin} is based on ``Meyn-Tweedie-type'' arguments which are not built in view of precise estimates.}

\smallskip
\noindent {\textbf{Criterions based on Wasserstein convergence.}}  {In this part, we provide tractable conditions which allow to use \cref{cor:Contraction}, based on Wasserstein convergence to equilibrium. {Oppositely to those of \cref{prop:weakreverting}}, these criteria have the advantage of being more likely to lead to quantitative bounds (a subject that we do not explore here). Furthermore, $\sigma$ will not be supposed to be bounded and we will include situations where all the moments are not uniformly controlled.}
 Before  giving {such conditions} in~\cref{thm:diffusionsconfluentes}  {we thus first recall in the theorem below two important results on Wasserstein-contraction} of solutions of SDEs.
\begin{thm} \label{thm:WangEberle} Assume $b$  and $\sigma$ are locally Lipschitz continuous and satisfy~\eqref{eq:onR+}.\\
\smallskip
\noindent $(a)$ {\it Uniform $L^q$-contraction}. Let {$q\! \in (1,+\infty)$}.  Assume that, for every $x$, $y\!\in \R^d$, $x\neq y$,
\begin{align}\label{eq:Lq-contraction} 
   \langle b(x)-b(y)\,|\,x-y\rangle +\tfrac 12\ \|\s(x)-\s(y)\|^2_F+ (\tfrac q2-1)\frac{|(\s(x)-\s(y))^{\top}(x-y)|^2}{|x-y|^2} \le   - \bar  \kappa_q |x-y|^2
\end{align}
where $\bar \kappa_{{q}}$ is positive constant.  Then, for every $\mu_i \!\in{\cal P}_q(\R^d)$, $i=1,2$, 
\begin{equation}\label{eq:Lq-confluence} 
\forall\, t\ge 0, \quad {\cal W}_{{q}}(\mu_1P_t, \mu_2P_t) \le e^{-\bar \kappa_{{q} t}}  {\cal W}_q(\mu_1, \mu_2). 
\end{equation}
In particular, for every $x$, $y\!\in \R^d$, for every $t\ge 0$, ${\cal W}_1(P_t(x,dy), P_t(x',dy)) \le e^{-\bar \kappa_1 t}  |x-y|$.
%Moreover, if $q\ge 2$, SDE~\eqref{eq:EDSbis} has a unique invariant distribution $\nu$ and $\nu \!\in {\cal P}_q(\R^d)$.

\smallskip
\noindent $(b)$ {\it $L^1$-contraction, elliptic case (see~\cite[Theorems~2.5  and 2.6]{Wang2020})}. Assume furthermore that $\sigma$ is uniformly elliptic (see \eqref{eq:sigmaUellip} with $d'\ge d$ to make ellipticity possible).  Assume furthermore that $b$ and $\s$ satisfy the following $L^1$-contraction assumption:   for every $x$, $y\!\in \R^d$, $x\neq y$,
\begin{align}
\nonumber \langle b(x)-b(y)\,|\,x-y\rangle +\tfrac 12\bigg( \|\s_0(x)-\s_0(y)\|^2_F-&\frac{|(\s(x)-\s(y))^{\top}(x-y)|^2}{|x-y|^2}\bigg) \\
\label{eq:L1-contractionelliptic} & \qquad\quad \le (\underline \kappa -(\underline \kappa +\bar  \kappa)\mbox{\bf 1} _{\{|x|\ge R\}})|x-y|^2
\end{align}
for some non-negative constants $\underline \kappa, \bar \kappa,R$ with $\bar \kappa>0$. Then there exists positive real constants $C,\lambda>0$ such that, for every $\mu_i \!\in{\cal P}_1(\R^d)$, $i=1,2$, 
\begin{equation}\label{eq:Lq-confluenceelliptic} 
\forall\, t\ge 0, \quad {\cal W}_1 (\mu_1P_t, \mu_2P_t) \le  C e^{-\lambda t}  {\cal W}_1 (\mu_1, \mu_2).
\end{equation} 
\noindent In particular, for every $x$, $y\!\in \R^d$, for every $t\ge 0$, ${\cal W}_1(P_t(x,dy), P_t(x',dy)) \le C e^{-\lambda t}  |x-y|$.
%Moreover, if $q\ge 2$, SDE~\eqref{eq:EDSbis} has a unique invariant distribution $\nu$ and $\nu \!\in {\cal 
 \end{thm}
 
% \noindent \tcr{@Gilles: transformer la rq en corollary pour avoir  $\ATROISter$ et conclure que l'on a les bonnes vitesses avec $\pi_t= \frac 12 (\mu +\nu)$, $\b= \frac 12(1-\frac 1q)$, etc. Adapter la reuve du thm accordingly en invoquant les bonnes conditions etc? Peut servir dans le pb de Fabien car sans doute m\^emes vitesses...}
 \begin{rem} $\bullet$ A  less sharp but simpler criterion  for~\eqref{eq:Lq-contraction} is
 \begin{align*}
\forall\, x,\, y\!\in \R^d, \qquad    \langle b(x)-b(y)\,|\,x-y\rangle +\tfrac{(q-1)\vee 1}{2}\|\s(x)-\s(y)\|^2_F
%+ (\tfrac q2-1)&\frac{|(\s(x)-\s(y))^{\top}(x-y)|^2}{|x-y|^2}\bigg) \le  
 \le - \bar  \kappa_q |x-y|^2
\end{align*}
since $\frac{|(\s(x)-\s(y))^{\top}(x-y)|^2}{|x-y|^2}\le \|\s(x)-\s(y)\|^2_F$.

\noindent {$\bullet$ When $q\ge 2$ in the above Claim~$(a)$, one checks that the uniform contraction assumption~\eqref{eq:Lq-contraction} implies the  existence of a unique invariant distribution $\nu$ for~\eqref{eq:EDSbis}, $\nu$ lying in $\mathcal{P}_q(\R^d)$ since the mean-reverting Hajek's criterion, e.g. with $|\cdot|^q$ as a Lyapunov function, is satisfied, namely
 \begin{equation}\label{eq:Hajekdiff}
({\rm \bf Haj)}_q\quad \forall\, x\!\in \R^d\setminus\{0\},Ê\quad  \langle b(x)\,|\, x\rangle + \tfrac12 \|\sigma(x)\|^2_{_F}+ (\tfrac q2-1) \frac{|(\s(x)^{\top}x|^2}{|x|^2}\le
\underline \kappa'   -\bar \kappa' |x|^2 
 \end{equation}
where $\bar \kappa'>0$. When $q\!\in [1,2)$, this no longer true and we have to add an extra mean-reverting assumption, e.g. still of the Hajek type, this time with the Lyapunov function $(1+|x|^2)^{\frac q2}$ (to avoid singularities). }%\tcr{@Gilles: remplacer par $r$ ?}
 \end{rem}
 
 %$\blacktriangleright$  Let $q\!\in [1,2)$. If, for every $x\!\in \R^d$,
% \begin{equation}\label{eq:Hajekdiff}
% \langle b(x)\,|\, x\rangle + \tfrac12 \|\sigma\|^2_{_F}+ (\tfrac q2-1) \frac{|(\s(x)^{\top}x|^2}{|x|^2}\le -\bar \kappa' |x|^2 +\underline \kappa' |x|^{(2-q)^+}
% \end{equation}
% for some real constants $\bar \kappa'$, $\underline \kappa'$, $\bar \kappa'>0$, then  EDS~\eqref{eq:EDSbis}  has a unique invariant distribution  $\nu$ and $\nu \!\in {\cal P}_q(\R^d)$. The above condition implies that the infinitesimal generator ${\cal L}$ of the the SDE satisfies  ${\cal L}(|\cdot|^q)\le \beta_q -\a_q |\cdot|^q $ which in turn ensures the existence of (at least) one  invariant distribution $\nu$ and that any such invariant distribution lies in ${\cal P}_q(\R^d)$. Under the assumptions of the above theorem uniqueness obviously holds  owing the confluence property. However, note that when $q\!\in [1,2)$,~\eqref{eq:Hajekdiff}  cannot be deduced from the above contraction assumptions~\eqref{eq:Lq-contraction} or~\eqref{eq:L1-contractionelliptic}  whereas such is the case for~\eqref{eq:Lq-contraction} when $q\ge 2$.  
% 
 \begin{thm} [Tractable conditions for a fixed $q\!>\!1$] \label{thm:diffusionsconfluentes} $(\alpha)$ {\em About  Condition~$\ATROISter$, see~\eqref{eq:W1boundbyter}}.  Assume that either all assumptions of  $(a)$ or all those of  $(b)$ from Theorem~\ref{thm:WangEberle} are in force. {Moreover, assume  that  {(${\rm \bf Haj})_q$ holds} for some $q>1$}.
 % the following  Hajek condition holds}
% %\tcr{@Gil: bizarre qd m\^eme \`a cause du comportement en $0$ notamment quand $q\!\in [1,2)$. Fonctionne alors sous l'hypo suppl\'ementaire  $%\P(\exists t: X_t=0)=0$}
%  \begin{align}\label{eq:Hajekdiff}
% ({\rm \bf Haj})_q\quad  \langle b(x)\,|\, x\rangle + \tfrac12 \|\sigma(x)\|^2_{_F}+ (\tfrac q2-1) \frac{|(\s(x)^{\top}x|^2}{{1+|x|^2}}\le\underline \kappa' -\bar \kappa' (1+ |x|^2)^{\frac{q}{2}} \end{align}
%holds for  some real constants $\bar \kappa'$, $\underline \kappa'$ with  $\bar \kappa'>0$. \tcb{Then~\eqref{eq:EDSbis} admits a unique invariant distribution $\nu$ with $\nu\!\in \mathcal{P}_q(\R^d)$}.  \tcr{@Fabien : plus clair pur moi o\`u \c ca sert: Assume that $b$ and $\s$ are $C^{1+\ve}_b$~(\footnote{i.e. have bounded and $\ve$-H\"older-continuous partial derivatives.})}.
% pour l'existence du flot tangent}. 
{Then,  SDE~\eqref{eq:EDSbis}  admits a unique invariant distribution $\nu$, $\nu \!\in {\cal P}_q(\R^d)$, and for any $\mu_0\!\in {\cal P}_q(\R^d)$ Condition~$\ATROISter$  is satisfied with $q$ and 
$$
\Psi(x,y)= |x-y|\quad \textnormal{and}\quad e(t) = e^{-\kappa t} \quad\textnormal{with $\kappa=\bar \kappa_q$ for $(a)$ and  $\kappa=\lambda$ for $(b)$}.
$$}
 %\smallskip
%$ (\beta)$  SDE~\eqref{eq:EDSbis} has a unique invariant distribution  $\nu$, $\nu \!\in {\cal P}_q(\R^d)$ and for any $\mu_0\!\in {\cal P}_q(\R^d)$,   $\ATROISter(ii)$ is satisfied with $q$, hence $\ATROISbis(ii)$ \tcb{from \cref{prop:markovcontraction}.}
 % \tcb{as well as $\ATROISter(ii)$  for $q$}. 
 \noindent $(\beta)$ {(About $\LSF$, see~\eqref{eq:LipStrgFeller}}. Moreover, if  $d'=d$, {if $b$ and $\sigma$ are ${\cal C}^1$ with bounded derivatives and if $\s$ is uniformly elliptic (condition already required in \cref{thm:WangEberle}$(b)$)} %assumption
%\begin{equation}\label{eq:sigmaUellip}
%\exists\, \underline \sigma_0>0 \mbox{ such that } \forall\, x,\, \xi \!\in \R^d,Ê\quad  \xi^{\top} \sigma\sigma^{\top} (x)\xi = |\s^{\top}(x)\xi|^2 \ge \underline \sigma_0^2 |\xi|^2,
%\end{equation} 
then Condition~$\LSF$ of Proposition~\ref{prop:markovcontraction} is satisfied for any $\theta_0>0$.

\smallskip

\noindent {$(\gamma)$  If the above assumptions in $(\alpha)$ and $(\beta)$ are satisfied and $\mu_0\!\in {\cal P}_q(\R^d)$, then the rates established in \cref{thm:abstrait} apply for $\E\,\mathcal{W}_p(\nu_t,\nu)^p$ with $0<p<q$.}
\end{thm}
\noindent The proof of this theorem and its  corollary hereafter are postponed to~\cref{subsec:proof2.10-2.11}.  A more striking version   of this result can be written under a more stringent control of the growth of~$\s$.

\begin{cor}[When all power moments are finite]\label{cor:A tous moments}   (a) If some positive $\kappa_1$, $\kappa_2$ and $C_\s$ exist such that
\begin{equation}\label{eq:corcontsbn}
\langle b(x),x\rangle \le \kappa_1-\kappa_2 |x|^2 \quad\textnormal{and} \quad  \|\s(x)\|_{_F} \le C_\s (1+|x|)^{1-\frac r 2}\quad \textnormal{for a given $0<r\le \frac{1}{2}$},
\end{equation}
then $({\rm \bf Haj)}_q$ holds for every $q>1$ and for every distribution $\mu_0$ such that $\mu_0 (e^{\lambda |\cdot|^r})<+\infty$, one has 
\[
\forall\, q\ge 1,\quad \sup_{ t\ge 0}\mu_0P_t(|\cdot|^q)\le C_{q,r}\sup_{ t\ge 0}\mu_0P_t(e^{\lambda |\cdot|^r})<+\infty. 
\] 
(b) If furthermore, all the other assumptions of \cref{thm:diffusionsconfluentes}$(\alpha)$ and $(\beta)$ are in force, 
$$
\big\| {\cal W}_p(\nu_t,\nu) \big\|_p\lesssim t^{-\frac{1}{2p}\wedge\left(\frac{1}{d}\right)^{-}}\quad \textnormal{in the sense of~\eqref{eq:toutunbis}}.
$$

\end{cor}

\subsubsection{Fractional and  Gaussian driven SDEs}\label{sec:fracsdes}
In this section, we emphasize that the bounds of~\cref{thm:abstrait} may apply to the non-Markovian setting. We here consider the case of  SDEs driven by  Gaussian processes with stationary increments in the additive case ($i.e.$ when the ``diffusion'' coefficient is constant), {including SDEs driven by a fractional Brownian motion (fBm)}. Let $\sigma$ denote an invertible matrix. We consider the following SDE
\begin{equation}\label{fracSDE}
 dX_t=b(X_t) dt+\sigma dG_t,
 \end{equation}
where  $(G_t)_{t\ge0}$ is a continuous Gaussian process with stationary increments  which  admits the following moving-average representation:
\begin{equation}\label{eq:moving-average}
 G_t=\underbrace{\int_{-\infty}^0 g(t-u)-g(-u) dW_u}_{\bar{G}_t}+\underbrace{\int_0^t g(t-u) dW_u}_{\tilde{G}_t},\quad t\ge0,
 \end{equation}
where $(W_t)_{t\in\ER}$ denotes a two-sided ($\ER^d$-valued) Brownian motion and $g:(0,+\infty)\mapsto[0, +\infty)$ is a  measurable function satisfying (at least)
 \begin{align*}
\forall t\in\R_+,\quad \int_{0}^{+\infty} |g(t-u) - g(-u)|^2 du <\infty .
\end{align*}
Note that in short, this representation writes:
\begin{equation}\label{eq:def_noise}
\forall t\in\R_+,\quad G_t = \int_{-\infty}^0 g(-u) \left(d W_{t+u} - d W_u\right).
\end{equation}
It is well-known that this representation holds for ``almost'' all continuous Gaussian processes with stationary increments (it holds true when the process is \textit{purely non deterministic}, see \textit{e.g.}~\cite[Theorem 3.5]{hida} for details). In particular,  it holds for the fractional Brownian motion with:
$$ g(t)=t^{H-\frac{1}{2}}.$$
Note that when $H=1/2$, $i.e.$ when $g\equiv 1$, one retrieves the classical Brownian motion ($H=1/2$). In the sequel of this section, we assume for $H\in(0,1]$ and  $\zeta>3/2$ that

%\noindent $\HFBM$: 
%\begin{itemize}\item[$(i)$] $g$ is ${\cal C}^2$ on $(-\infty,0)$,
%\item[$ (ii) $]  $\exists t_0>0, \forall t\in(0,t_0], \quad g(t)= t^{H-\frac{1}{2}}$,
%\item[$(iii)$] $\exists C>0$ such that   $\forall u\in [1,+\infty)$,  $|g''(u)|\le C|u|^{-\zeta}.$
%\end{itemize}
\smallskip
\noindent $\HFBM$: $\begin{cases}
(i) &\textnormal{ $g$ is ${\cal C}^2$ on $(0,+\infty)$,}\\
(ii) &\textnormal{  $\exists t_0>0, \forall t\in(0,t_0], \quad g(t)= t^{H-\frac{1}{2}}$,}\\
(iii) &\textnormal{$\exists C>0$ such that   $\forall u\in [1,+\infty)$,  $|g''(u)|\le C|u|^{-\zeta}.$}
\end{cases}$

%\end{itemize}
%\begin{equation}
%|g''(u)|\le C|u|^{-\zeta}\quad \textnormal{if $u\in [t_0,+\infty)$.}$$
\smallskip For the fBm with Hurst parameter $H\in(0,1)$, $\HFBM$ is satisfied with $\zeta=\frac{5}{2}-H$ (Note that this trivially holds for all $\zeta>0$ when $H=1/2$). 
%\begin{rem} 
%The reader may be surprised by the condition 
%\end{rem}
 We also assume the following contraction condition with parameters $\kappa\in(0,+\infty)$ and $R,\lambda\in [0, +\infty)$:

\noindent $\mathbf{(S_{\kappa,R,\lambda})}$: $b$ is Lipschitz continuous and  $\forall x,y\in\ER^d$, $\big\langle b(x)-b(y),x-y\big\rangle\leq \begin{cases} -\kappa |x-y|^2& |x|,|y|\ge R\\
\lambda |x-y|^2& \textnormal{otherwise}.
\end{cases}
$\\

\noindent \noindent \textbf{Background on ergodicity of SDEs driven by stationary Gaussian processes.} Since~\cite{Hairer2005}, it is now well-known (at least when $(G_t)_{t\ge0}$ is a fBm) that even if it is certainly not Markovian, the solution to~\eqref{fracSDE} can be cast as the marginal of an infinite-dimensional Feller (homegeneous) process $Z_t:=(X_t,(W_{s+t})_{s\le 0})$ where $W$ denotes the two-sided Wiener process involved in~\eqref{eq:moving-average}. This process $Z$ takes values in  $\ER^d\times {\cal H}_{H}$ where ${\cal H}_{H}$ is a H\"older-type space the supporting the Wiener measure on $\ER_{-}$. {We call \textit{generalized initial condition} a probability $\Pi_0$ on $\ER^d\times {\cal H}_{H}$. Then, an invariant distribution $\Pi$ is a generalized initial condition which is (classically) invariant by the transitions of $(Z_t)_{t\ge0}$.}  The result below is due  to~\cite{Hairer2005} (see also~\cite{panloup-richard} for the extension to Gaussian processes).
\begin{prop} Assume $\HFBM$ with $\zeta>3/2$ and $\mathbf{(S_{\kappa,R,\lambda})}$ for some $\kappa, R,\lambda\in(0,+\infty)\times[0, +\infty)^2$.  Then,  $(Z_t)_{t\ge0}$ admits a unique~(\footnote{In this setting, one says that uniqueness holds if the distribution of the stationary induced process $(X_t)_{t\ge0}$ is unique. This implies in particular that $\nu$ is unique.}) invariant distribution $\Pi$. {Its first marginal, denoted by $\nu$ in the sequel, has moments of any order}.
\end{prop}
In~\cite{Hairer2005}, that we refer to for more detailed definitions, the ergodicity (in total variation) is also proved under $\mathbf{(S_{\kappa,R,\lambda})}$ but in the general setting, the memory of the process combined with its roughness when $H$ is small, leads to very small rates of the order $t^{-\beta_H}$ with $\beta_H\in(0,1)$ (see ~\cite{hairer-pillai,Fontbona2017,DPT} for extensions to the multiplicative setting).

\medskip
\noindent  In~\cite{panloup-richard} and then in~\cite{sieber-li}, some ergodic results have been obtained in a {more friendly setting where the parameter $\lambda$ of $\mathbf{(S_{\kappa,R,\lambda})}$ is small or equal to $0$}. The result below is an adaptation of~\cite[Theorem 1.3]{sieber-li} to our assumptions {(the proofs  of all the statements of this section  are postponed to~\cref{annexproofTVfrac})}. {In all the following results, we assume that $(X_t)_{\ge0}$ has a generalized initial condition $\Pi_0$ whose first marginal has moments of any order so that $\E |X_0|^q<+\infty$ for all $q\ge 1$.}
\begin{prop}\label{prop:TVfrac} Assume $\HFBM$ with $\zeta>3/2$. 
%~(\footnote{By generalized initial condition, we mean a distribution on the product space $\ER^d\times {\cal H}_{H}$.}) 
For every $(\kappa,R)\in(0, +\infty)\times[0, +\infty)$, there exists $\lambda_0>0$ such that {if $\mathbf{(S_{\kappa,R,\lambda})}$ holds with $\lambda\le \lambda_0$}, there exists some positive  $\rho$ and $C$ such that
$$\|{\cal L}(X_t)-\nu\|_{TV}\le C e^{-\rho t}.$$
%$$\|{\cal L}(X_t|{\cal F}_s)-\pi\|_{TV}\le C e^{-\rho (t-s)} \int |X_s-y|\pi(dy).$$
\end{prop}
%{The proof of this proposition and of all the results of this section are achieved in~\cref{annexproofTVfrac}.}
In the above result (and in the following ones), the main case of application is the ``convex'' case, which corresponds to $\lambda=0$ ({When $b=-\nabla U$, $\mathbf{(S_{\kappa,R,0})}$ holds true if $U$ is convex on $\ER^d$ and uniformly strongly convex outside a compact set}). The fact that $\lambda_0$ is positive means that it extends to ``slightly'' non-convex setting. By ``slight'', we mean that the lack of convexity is not sufficiently strong to prevent a contraction by synchronous coupling. 

\smallskip
\noindent {\textbf{Rate of conditional distributions and application of~\cref{thm:possiblynonmarkov}.}} Oppositely to the Markovian case, the exponential convergence to $\nu$ of the law of $X_t$ {recalled above does not extend to that of the conditional distributions $({\cal L}(X_t|{\cal F}_s))_{t\ge s}$} (in the Markovian setting, one expects $\|{\cal L}(X_t|{\cal F}_s)-\nu\|_{TV}$ to be of the order $e^{-\rho(t-s)}$). The following theorem is the cornerstone of the main result stated right after.
\begin{thm}\label{thm:fractrao} Assume $\HFBM$ with $\zeta>3/2$. 
%~(\footnote{By generalized initial condition, we mean a distribution on the product space $\ER^d\times {\cal H}_{H}$.})
 For every $(\kappa,R)\in (0, +\infty)\times[0, +\infty)$, there exists $\lambda_0>0$ such that {if $\mathbf{(S_{\kappa,R,\lambda})}$ holds with $\lambda\le \lambda_0$}, then for every  $\varepsilon>0$,  there exists a functional $\Upsilon$ defined on $\ER^d\times{\cal C}((-\infty,0],\ER^d)$ such that for every $0\le s\le t$, $\P$-$a.s.$, 
$$
\|{\cal L}(X_t|{\cal F}_s)-\nu\|_{TV}\le \Upsilon(X_s,(W_{s+t})_{s\le 0}) (1\vee (t-s))^{-\zeta+\frac{3}{2}+\varepsilon}
$$
and $\Upsilon(X_s,(W_{s+t})_{s\le 0})$ has {finite} moments at any order $q$. 
\end{thm}
%The proof of~\cref{thm:fractrao} is postponed to~\cref{annexproofTVfrac}. 
\noindent The above result matches with Assumption $\ATROIS$ of~\cref{thm:possiblynonmarkov} which leads to the following theorem.
\begin{thm}[General case]\label{cor:fractrao} Let the assumptions of~\cref{thm:fractrao} be in force with $\zeta>{3/2}$. Then, 
\begin{itemize}
\item Very short Memory: If $\zeta>{5/2}$, $\ATROIS$ holds true for any $q\ge 1$ with $\gamma=1/2$. Thus, for any $q > 1$,  $\AUN$ holds for every $t_0>0$ with $\beta=\frac{1}{2}(1-\frac 1q)$, $\pi_t=\frac{1}{2}(\bar\nu_t+\nu)$ and $e\!\in \mathbb{L}^1([0,+\infty))$ defined by $e(t)=(1\vee t)^{-\zeta+\frac{3}{2}+\varepsilon}$ with $\varepsilon\in(0, \zeta-\frac{5}{2})$. Thus, 
$$\big\| {\cal W}_p(\nu_t,\nu) \big\|_p\lesssim t^{-\frac{1}{2p}\wedge\left(\frac{1}{d}\right)^{-}} \quad \textnormal{in the sense of~\eqref{eq:toutunbis}}.$$
\item Short and long Memory:  If $\zeta\in{(3/2,5/2})$, $\ATROIS$ holds true for any $q\ge 1$ {with $\gamma=\frac{1}{2}(\zeta-\frac{3}{2}-\ve)$ (with $\ve\in (0, \zeta-\frac{3}{2})$)}. Thus, for any $q > 1$ and $\gamma=\frac{1}{2}(\zeta-\frac{3}{2}-\ve)$,  $\AUNgamma$ holds for every $t_0>0$ with $\beta=\frac{1}{2}(1-\frac 1q)$, $\pi_t=\frac{1}{2}(\bar\nu_t+\nu)$ and $e$ defined by $e(t)=(1\vee t)^{-\zeta+\frac{3}{2}+\varepsilon}$ with $\varepsilon\in(0, \zeta-\frac{3}{2})$. Thus, for any $\ve>0$ a constant $C_{\varepsilon}$ exists such that
$$
\big\| {\cal W}_p(\nu_t,\nu) \big\|_p\le C_\varepsilon t^{-(\zeta-\frac{3}{2})\left(\frac{1}{2p}\wedge\frac{1}{d}\right)-\ve} .
$$
%{@Fabien: V\'erifier que \dots Moreover the rates of convergence of $\,{\cal W}_p^p(\nu_t,\nu)$ satisfy~\eqref{eq:boundsatoutmoment} i.e. those obtained in Corollary~\ref{cor:A tous moments}.}
\end{itemize}
\end{thm}
\begin{cor}[Fractional SDEs]\label{cor2:fracappl} {Assume that $(G_t)_{t\ge0}=(B_t^H)_{t\ge0}$, $i.e.$ is a fractional Brownian motion with Hurst parameter $H\in(0,1)$.} Then for every $(\kappa,R)\in(0, +\infty)\times[0, +\infty)$, there exists $\lambda_0>0$ such that {if $\mathbf{(S_{\kappa,R,\lambda})}$ holds with $\lambda\le \lambda_0$}, then for every  $\varepsilon>0$,  
%the previous statement applies with $\zeta=\frac 52-H.$ 
a constant $C_\varepsilon$ exists such that
$$\big\| {\cal W}_p(\nu_t,\nu) \big\|_p\le C_\varepsilon t^{-(1-H)\left(\frac{1}{2p}\wedge\frac{1}{d}\right)-\ve} .$$
\end{cor}

\begin{rem} $\blacktriangleright$  {Since $\HFBM$ holds with $\zeta=\frac{5}{2}-H$ for the fBm, ~\cref{cor2:fracappl} is a direct application of~\cref{cor:fractrao} in the case $\zeta\in{(3/2,5/2})$. 
This explains the terminology ``very short memory'' which emphasizes that  the fBm never falls in the case $\zeta>5/2$, even when $H<1/2$ (which is usually considered as a ``short memory'' setting in the literature). Nevertheless, the case $\zeta>5/2$ remains interesting for applications since it provides a setting where one may have the local behavior of a fBm but with a memory sufficiently small to preserve the rate orders of the Markovian setting.}
 %Thus, this theorem allows to consider processes which have the local behavior of the fBm but a shorter memory. For the case of fractional Brownian motions, we refer to ~\cref{sec:fBmres} below.

\smallskip
\noindent $\blacktriangleright$ Note that in $\HFBM(ii)$, we impose $g$ to be equal to $t^{H-\frac{1}{2}}$ near $0$. The extension to the case where $g$  is only close to $t^{H-1/2}$ is not clear. Actually, the $TV$-bound in ~\cref{thm:fractrao} is based on a non-trivial coupling argument which requires to invert the kernel related to the process. This inversion is possible {when $g$ coincides with the fractional kernel near $0$} but it is not clear that it extends to any function $g$ (see Condition $\mathbf{(C3)}$ of ~\cite{panloup-richard} for further details on this topic).
\end{rem}

\noindent {\textbf{The specific case of Gaussian stationary processes.}} When $b$ is an affine function, the process (denoted by $(Y_t)_{t\ge0}$ in the sequel) is Gaussian and in this case, it is possible to estimate ${\rm Cov}(f(Y_t),f(Y_s))$ for a given bounded measurable function $f$ without resorting to~\cref{thm:fractrao}, but 
using semi-explicit computations of ${\rm Cov}(Y_t,Y_s)$ combined with Hermite expansions (see~\cref{lem:gaussian}). When applies, this approach may lead to better bounds than the ones obtained in Corollary~\eqref{cor2:fracappl}. For instance, the following result holds for the fractional Ornstein-Uhlenbeck process:

 \begin{thm}\label{thm:specou} Let $(Y_t)_{t\ge0}$ denote the stationary one-dimensional fractional Ornstein-Uhlenbeck  process solution to:
$$ dY_t=-\lambda Y_t dt+\sigma dB_t^H,$$
where $\lambda$ and $\sigma$ are positive numbers. Then, for any $H\in(0,1)$, $\AUNgamma$ holds for any $q>0$, with $\pi_t=\nu$ (where $\nu$ denotes the-first marginal of-the invariant distribution of $(Y_t)_{t\ge0}$), $\beta=1/2$ and  
$$\gamma=\begin{cases} \frac{1}{2}&\textnormal{if $H<1/2$,}\\
1-H &\textnormal{if $H>1/2$}.
\end{cases}
$$
Thus,  
$$\big\| {\cal W}_p(\nu_t,\nu) \big\|_p\lesssim t^{-\frac{1}{p}(\frac 12\wedge\left(1-H\right))}.$$
%\tcb{Thus, for any $p>0$, $\|{\cal W}_p(\nu_t,\nu)\|_p$ satisfies the estimate~\eqref{eq:toutun} with $\beta=1/2$.}
\end{thm}
\begin{rem} $\blacktriangleright$ {When $H>1/2$, the result ``only'' allows to remove the $\varepsilon$ in \cref{cor2:fracappl}, $i.e.$ to replace the exponent $1-H-\varepsilon$ by $1-H$. When $H<1/2$, there is a real gain since we replace $1-H-\varepsilon$ by $1/2$. }

\noindent $\blacktriangleright$ {This theorem may extend to the multidimensional case by using that the coordinates are independent (in this very specific case) and that, for two probabilities $\mu$ and $\nu$ on $\ER^d$ with marginals $(\mu_i)_{i=1}^d$ and $(\nu_i)_{i=1}^d$, ${\cal W}_p(\mu,\nu)\le \sum_{i=1}^d {\cal W}_p(\mu_j,\nu_j)$. this would lead to 
$$\big\| {\cal W}_p(\nu_t,\nu) \big\|_p\le C d t^{-\frac{1}{p}(\frac 12\wedge\left(1-H\right))}$$
where $C$ is independent of $d$.}

\noindent $\blacktriangleright$ Note that the invariant distribution $\nu$ is a Gaussian centered distribution with variance $\sigma_H^2=\frac{\sigma^2}{\lambda^{2H}} H\Gamma(2H)$ (it can be deduced from~\cite[Remark 2.4]{cheridito03} using the fact that $\int_0^{+\infty}  \frac{x^{1-2H}}{1+x^2} dx=\frac{\pi}{2 \sin(\pi H)}$).
\end{rem}

\section{Proof of  Theorem~\ref{thm:abstrait}}
\label{sec:proofmain}
We need the following technical lemma, adapted from~\cite[Proof of Theorem~1, Step~1]{FG}.
\begin{lem}\label{lem:tech}Let $p>0$,  $\b\!\in (0, 1/2]$. Let $t>0$ be fixed and let $L: (0, +\infty)\to \R_+$ be defined by 
\[
L_t(u):=\sum_{\ell\ge 0} 2^{-p\ell} \min\Big(u^{\frac{1}{2\b}}, \big(u/t)^{\frac 12}2^{d\ell(1-\b)}\Big).
\]

The function $L$ satisfies the following upper-bounds depending on $p$, $\b$ and the dimension $d$ where $C_{p,\b,d}>0$ denotes    a positive constant only depending on $p$, $\b$, $d$ that may vary from line to line.
\begin{itemize}
\item If $p>d(1-\b)$, then 
\[
L_t(u) \le C_{p, \b,d} \min\Big( u^{\frac{1}{2\b}}, \Big( \frac{u}{t}\Big)^{\frac 12}\Big).
\]
%\item  If $p = d(1-\b)$, then 
%$$
%L(u) \le C_{p,\b,d} \min\bigg( u^{\frac{1}{2\b}},\Big(\frac{u}{t}\Big)^{1/2}\log\big(1+ ut^{\frac{\b}{1-\b}}\big)\bigg).
%$$
\item  If $p=d(1-\b)$ then
	\begin{align}
	\label{eq:p=d(1-beta)a} L_t(u) &\le C_{p,\b,d}
	%\tcr{\Big( \frac{u}{t}\Big)^{\frac 12}}  \big(1+\big(\log(u\,t^{\frac{\b}{1-\b}})\big)^+\big) 
	 \Big(u^{\frac{1}{2\b}} \mbox{\bf 1}_{\{u\,t^{\frac{\b}{1-\b}}  \le 1\}}+ \Big(\frac ut\big)^{\frac 12} \big(1+\log(u\,t^{\frac{\b}{1-\b}})\big) \mbox{\bf 1}_{\{ut^{\frac{\b}{1-\b}}  \ge 1\}}\Big)\\
		\label{eq:p=d(1-beta)b} & \le C_{p,\b,d}
	%\tcr{\Big( \frac{u}{t}\Big)^{\frac 12}}  \big(1+\big(\log(u\,t^{\frac{\b}{1-\b}})\big)^+\big) 
	\Big(\frac ut\Big)^{\frac 12} \big(1+\big(\log(u\,t^{\frac{\b}{1-\b}})\big)^+\big) 
	\end{align}
	%\tcr{F: Attention, je crois qu'il y avait une erreur dans cette in\'egalit\'e.}
%	which implies that
%	\[
%	L(u) \le C_{p,\b,d} \min\Big(u^{\frac{1}{2\b}}, \Big(\frac ut\Big)^{\frac 12} \Big)  \log(1+u\,t^{\frac{\b}{1-\b}}).
%	\]
%
\item If $p < d(1-\b)$, then
$$
L(u) \le C_{p,\b,d} u^{\frac{1}{2\b}}\min\Big( 1,\big(ut^{\frac{\b}{1-\b}}\big)^{-\frac{p}{2\b d}}\Big).
$$
\end{itemize}
\end{lem}
%\tcr{J'ai choisi de mettre directement la preuve du th\'eo 2.3 (ce qui \'evite la r\'ep\'etition de proof of th\'eo 2.3) et de faire r\'ef\'erence au lemme technique à la suite avec sa preuve. N'h\'esite pas à rechanger.} 
\paragraph{Proof of~\cref{thm:abstrait}.} 
	%\textcolor{red}{@Gil: Warning la cste depend de $\eta$ avec $t\ge \eta$}. 
	We consider successively the cases $\gamma=\frac 12$ and $\gamma\in(0,1/2)$. Note that the constants with capital letter $K$ may vary from line to line.
	
\noindent \textbf{Case $\gamma=\frac 12$.} The proof is based  on the upper-bound~\eqref{eq:UpperW_p}
 for the mean ${\cal W}_p$-distance and~\cref{lem:tech} applied with $u= 2^{-2 n \b q}\!\in (0,1)$, $n\ge 0$ to derive the bound for $t\ge\tzero$, 
	\begin{equation}\label{eq:meanWL1bis}
	\E\, {\cal W}_p(\nu_t,\nu)^p \le K_{\tzero,\b,p,q,d} \sum_{n\ge 0} 2^{pn}L_t(2^{-2 n \b q}),
	\end{equation}
	where $K_{\tzero,\b,p,q,d}= K_{\tzero, p,\b,d} {\Mpiq}$ for convenience  throughout the proof. 
	
	We will inspect successively all the cases depending on $p$, $d$, $\b$ and their sub-cases depending on $q$ if necessary.
	
	\smallskip
	\noindent $\blacktriangleright$ {\em Case $p>d(1-\b)$}. One derives from ~\cref{lem:tech} that
	\[
	\E\, {\cal W}_p(\nu_t,\nu)^p \le K_{\tzero,p,\b,d,q} \sum_{n\ge 0} 2^{pn}\min\Big(2^{-nq} , \frac{2^{-n\b q}}{\sqrt{t/\tzero}} \Big).
	\]
	(up to a change of the constant when $\tzero\!\in (0,1)$ by a factor $1/\sqrt{\tzero}$).
	
	 $\quad$ -- If $q>\frac{p}{\b}$ then, setting  $ c_{p,q,\b} = \sum_n 2^{(p-\b q)n}<+\infty$ so that
	 \[
	\E\, {\cal W}_p(\nu_t,\nu)^p \le K_{\tzero,\b,p,q,d} c_{p,q,\b}t^{-\frac 12}.
	\] 
	
	 $\quad$ -- If $q\!\in \big(p, \frac{p}{\b}\big)$, we introduce the tipping index $\ds n_0(t):=n_{\tzero}(t)= \Big \lceil \frac{\log(t/\tzero)}{2(1-\b)q\log 2}\Big\rceil\ge 0$ since $t\ge\tzero$, which satisfies
	 \[
	 n\ge n_0(t) \quad \Longleftrightarrow\quad 2^{-nq} \le\frac{2^{-n\b q}}{\sqrt{t/\tzero}}.
	 \]
	 %\tcr{ok}
	 Consequently
	 \begin{align*}
	 \E\, {\cal W}_p(\nu_t,\nu)^p &\le K_{\tzero,p,\b,d,q} \left [\Big(\sum_{n=0}^{n_0(t)-1}2^{(p-\b q)n}\Big) t^{-\frac 12} + \sum_{n\ge n_0(t)} 2^{(p-q)n} \right]\\
	 					     & \le K_{\tzero,p,\b,d,q} \left [2^{(p-\b q)n_0(t)} t^{-\frac 12} +2^{(p-q)n_0(t)} \right].
	 \end{align*}
	 One easily checks that 
	 \[
	 t^{-\frac12}2^{(p-\b q)n_0(t)} \le c_{\tzero, p,q,\b}t^{\frac{p-\b q}{2q(1-\b)}-\frac 12}=  c_{\tzero, p,q,\b}t^{-\frac{q-p}{2q(1-\b)}}
	 \]
	 and that, up to the real constant, the same bounds holds for  	 $2^{(p-q)n_0(t)}$. Hence
	 \[
	  \E\, {\cal W}_p(\nu_t,\nu)^p \le K''_{\tzero,p,\b,d,q} t^{-\frac{q-p}{2q(1-\b)}}.
	 \]
	 
	\noindent $\blacktriangleright$ {\em Case $0<p<d(1-\b)$}. Inserting the estimate for $L(2^{-2\b q n})$ into~\eqref{eq:meanWL1bis}, we get, up to change of the constant $K_{\tzero,p,\b,d,q}$, that
	\[
	  \E\, {\cal W}_p(\nu_t,\nu)^p \le K_{\tzero,p,\b,d,q} \sum_{n\ge 0} 2^{(p-q)n} \min \Big(1, 2^{\frac{q pn}{d}}\big(t/\tzero\big)^{-\frac{p}{2(1-\b) d}}\Big).
	\]
	
	$\quad$ -- If $q>\frac{dp}{d-p}$, i.e. if  $p-q+\frac{qp}{d}<0$, it is clear that $c_{p,q,d}= \sum_n 2^{(p-q+\frac{qp}{d})n}<+\infty$. Hence 
	%\tcr{$\nu_t$ ou $\nu_t$ dans la suite ?}
	\[
	  \E\, {\cal W}_p(\nu_t,\nu)^p \le K_{\tzero,p,\b,d,q} c_{p,q,d} (t/\tzero)^{-\frac{p}{2(1-\b) d}}= K_{\tzero,p,\b,d,q}\,t^{-\frac{p}{2(1-\b) d}}.
	\]
	
	$\quad$ -- If $q\!\in (p,\frac{dp}{d-p})$, then one checks that $2^{\frac{q pn}{d}}\big(t/\tzero\big)^{-\frac{p}{2(1-\b) d}}\ge 1$ iff $n\ge n_0(t)$ (as defined in the previous case). Hence, elementary computations show that
	\[
	  \E\, {\cal W}_p(\nu_t,\nu)^p \le K''_{\tzero,p,\b,d,q} \big( 2^{(p-q)n_0(t)}+ 2^{(p-q+\frac{qp}{d})n_0(t)}(t/\tzero)^{-\frac{p}{2d(1-\b)}}\big)\le K_{\tzero,p,\b,d,q}^{(3)} t^{-(1-\frac p q)\frac{1}{2(1-\b)}}.
	\]

	\noindent $\blacktriangleright$ {\em Case $p=d(1-\b)$}.% \tcr{verifier si la petite typo sur l'inegalite du lemme n'a pas de consequences.} 
	
	$\quad $  -- If $q>\frac{p}{\b} =d(\frac 1 \b-1)$, using~\eqref{eq:meanWL1bis},~\eqref{eq:p=d(1-beta)b} from  Lemma~\ref{lem:tech} after noting that $1+ (\log x)^+\le \kappa_0 \log(2+x/ t_0^{\frac{\b}{1-\b}})$ for every $x\ge t_0$, we  get for every $t\ge t_0$
	\begin{align*}
	\E\, {\cal W}_p(\nu_t,\nu)^p &\le K_{\tzero,\b,p,q,d} \,t^{-\frac 12}\sum_{n\ge 0} 2^{(p-\b q)n}\log\big({2}+2^{-2\b n q}(t/t_0)^{\frac{\b}{1-\b}}\big)\\
	& \le K_{\tzero,\b,p,q,d} \,t^{-\frac 12}\log\big(2+(t/t_0)^{\frac{\b}{1-\b}}\big)\sum_{n\ge 0} 2^{(p-\b q)n
	%\log\big(\tcb{2}+2^{-2\b n q}(t/t_0)^{\frac{\b}{1-\b}}\big).
	}
	\end{align*}
	then $p-\b q<0$ so that, for every $t\ge \tzero$, 
	\[
	\E\, {\cal W}_p(\nu_t,\nu)^p \le K_{\tzero,p,\b,d,q} \frac{\log(2+(t/t_0)^{\frac {\b}{1-\b}})}{\sqrt{t}}\le K_{\tzero,p,\b,d,q} \frac{\log(1+t)}{\sqrt{t}}.
	\]
		
	$\quad$ -- If $q\!\in (p,\frac{p}{\b}\big)$ then $p-\b q>0$. We temporarily set $\theta= t/t_0$ to alleviate notation. Still using the tipping index $n_{0}(t)$ combined with~\eqref{eq:p=d(1-beta)a}  from  Lemma~\ref{lem:tech} 
	%$n_{0}(t)= \min\{n: 2^{-2\beta qn }(t/t_0)^{\frac{\beta}{1-\beta}}\le 1\}$, 
	we have
	\begin{align*}
	\E\, {\cal W}_p(\nu_t,\nu)^p& \le K_{\tzero,p,\b,d,q}\bigg(\sum_{n\ge n_{0}(t)} 2^{pn} (2^{-2\beta qn})^{\frac{1}{2\beta}}+t_0^{-\frac 12}\theta^{-\frac 12}\sum_{n=0}^{n_{0}(t)-1}2^{(p-\b q)n} \big(1+ \log\big(\underbrace{\theta^{\frac{\b}{1-\b}}2^{-2\b q n}}_{\ge 1}\big)\big)\bigg).
	%\\
	%& \le K'_{\tzero,p,\b,d,q}\Big(2^{(p-q)n_{0}(t)}+ t^{-\frac 12}2^{(p-\b q)n_{0}(t)} \theta\big(1+ \log\big(\theta^{\frac{\b}{1-\b}}\big)\big)\Big)
	\end{align*}
	First note that, using the definition of $n_0(t)$, 
	\[
	\sum_{n\ge n_{0}(t)} 2^{pn} (2^{-2\beta qn})^{\frac{1}{2\beta}} = \sum_{n\ge n_0(t)} 2^{(p-q)n}\le K_{\tzero,p,\b,d,q}2^{(p-q)n_0(t)}\le  K_{\tzero,p,\b,d,q} \theta^{-\frac{q-p}{2(1-\beta)q}}.
	\]
	Then note that 
	\begin{equation}\label{eq:ineqtech0}
	\theta^{-\frac 12}\sum_{n=0}^{n_{0}(t)-1}2^{(p-\b q)n} \le \theta^{-\frac 12}\frac{ 2^{(p-\b q)n_0(t)} }{2^{p-\b q}-1}\le K_{\tzero,p,\b,d,q}\theta^{-\frac{q-p}{2(1-\beta)q}}
	\end{equation}
	The last term to deal with ($n_0(t)\ge 1$ except if $t=t_0$ since $\theta = t/t_0\ge 1$) is as follows
	\begin{align*}
	\theta ^{-\frac 12}\sum_{n=0}^{n_{0}(t)-1}2^{(p-\b q)n}\log\big(\theta^{\frac{\b}{1-\b}}2^{-2\b q n}\big)&= \theta ^{-\frac 12}\sum_{n=0}^{n_{0}(t)-1}2^{(p-\b q)n}\Big(\frac{\b}{1-\b}\log \theta-2\b q n \log 2\Big)
	\end{align*}
	%One checks that both $2^{(p-q)n_{0}(t)}\le  \theta^{-\frac{q-p}{2(1-\beta)q}}$
	%\le  t ^{-\frac{q-p}{2(1-\beta)q}}$ 
	%and $ t^{-\frac 12}2^{(p-\beta q)n_{0}(t)} \le 2^{p-\b q}t_0^{-1/2}( t/t_0) ^{-\frac{q-p}{2(1-\beta)q}}$. 
%	Using that for every $y>0$ and $\ve \!\in (0,1)$, $\log(y) \le \frac{1}{\ve}\,y^{\ve}$, we derive that
%	%for every $t\ge \tzero$, 
%	\begin{align*}
%	t^{-\frac 12}\sum_{n=0}^{n_{0}(t)-1}2^{(p-\b q)n}  \log\big(\underbrace{\theta^{\frac{\b}{1-\b}}2^{-2\b q n}}_{\ge 1}\big)&\le c_{\b,p,q,\ve} t^{-\frac 12+\frac{\b \ve}{1-\b}}\sum_{n\ge 0} 2^{(p-\b q (1+2\ve))n}= c'_{\b,p,q,\ve} t^{-\frac 12+\frac{\b \ve}{1-\b}}.
%	\end{align*}
%	For $\ve >0$  small enough,  the constant $c'_{\b,p,q,\ve}$ is finite and $\frac 12-\frac{\b \ve}{1-\b}> \frac{q-p}{2(1-\b)}$ since since $q<p/\b$. 
By the definition of $n_0(t)$ we derive that $\log \theta \le n_0(t) 2(1-\b) q \log 2$ so that 
$$
\frac{\b}{1-\b}\log \theta-2\b q n \log 2\le  2\b q \log 2+ 2\b q  \log 2 (n_0(t)-1-n).
$$
Consequently, one gets 
	\begin{align*}
\theta ^{-\frac 12}\sum_{n=0}^{n_{0}(t)-1}2^{(p-\b q)n}\Big(\frac{\b}{1-\b}\log \theta-2\b q n \log 2\Big)& \le  2\b q \log 2\Big( \theta ^{-\frac 12}\sum_{n=0}^{n_{0}(t)-1}2^{(p-\b q)n}\\
&\quad + \theta ^{-\frac 12}\sum_{n=0}^{n_{0}(t)-1}2^{(p-\b q)(n_0(t)-1-n)}n\Big)
	\end{align*}
where  we reversed the indexation	in the second sum. Now 
	\begin{align*}
	\theta ^{-\frac 12}\sum_{n=0}^{n_{0}(t)-1}2^{(p-\b q)(n_0(t)-1-n)}n&= \theta^{-\frac 12}2^{(p-\b q)(n_0(t)-1)}\sum_{n=0}^{n_{0}(t)-1}2^{(\b q-p)n}n\\
	& \le  \theta^{-\frac 12}2^{(p-\b q)n_0(t)}\sum_{n\ge 0}2^{(\b q-p)n}n.
	\end{align*}
The series on the right-hand side is clearly finite since $\b q -p<0$ and $\theta^{-\frac 12}2^{(p-\b q)n_0(t)}$ has been treated with in~\eqref{eq:ineqtech0}. Finally, collecting all these bounds yields
	\begin{align*}
	\E\, {\cal W}_p(\nu_t,\nu)^p& \le 
	%K'_{\tzero,p,\b,d,q} t^{-(1-p/q)\frac{1}{2(1-\b)}}\le 
	 K_{\tzero,p,\b,d,q} t^{-\frac{q-p}{2q(1-\b)}}. 	%+ t^{-\frac 12}\sum_{n=0}^{n_0(t)-1} 2^{(p-\b q)n}   \log \big(2(t/\tzero)^{\frac{\b}{1-\b}}2^{-2\b q n}\big)\bigg)
	\end{align*}
%	since $\log(1+x)\le \log 2x$ if $x\ge 1$ and $(t/\tzero)^{\frac{\b}{1-\b}}2^{-2\b q n}\ge 1$ when  $0\le n\le n_0(t)-1$. Then taking advantage of what we did in the previous cases, we know that $ t^{-\frac 12}\log \big(2(t/\tzero)^{\frac{\b}{1-\b}}) \sum_{n=0}^{n_0(t)-1} 2^{(p-\b q)n}$ is bounded by $ t^{-(1-p/q)\frac{1}{2(1-\b)}} \log \big(2(t/\tzero)^{\frac{\b}{1-\b}}) $ up to a positive constant.
%	The last term to be inspected reads $t^{-\frac 12}R(t)$ with $R(t)= -2\b q \sum_{n=0}^{n_0(t)-1} 2^{(p-\b q)n} n$ is clearly negative. Finally
%	\[
%	\E\, {\cal W}_p(\nu_t,\nu)^p\le K''_{\tzero,p,\b,d,q} t^{-\frac{1-p/q}{2(1-\b)}}\log(1+ t).
%	\]
\textbf{Case $\gamma\in(0,1/2)$.} By~\eqref{eq:UpperW_p}, we have in this case:
\begin{equation*}
	\E\, {\cal W}_p(\nu_t,\nu)^p \le K_{\tzero,\b,p,q,d} \sum_{n\ge 0} 2^{pn}L_{t^{2\gamma}}(2^{- \beta n q}).
	\end{equation*}
	In other terms, if we set $\tau= t^\gamma$, we retrieve the right-hand member of ~\eqref{eq:meanWL1bis}. It follows that the bounds obtained in the case $\gamma=1/2$ extend to $\gamma\in(0,1/2]$ through the change of variable $\tau= t^\gamma$. The result follows.
	\hfill$\Box$
%\paragraph{A preliminary lemma.} We first need to establish 

%\paragraph{Proof of~\cref{lem:tech}.}

\begin{proof}[Proof of~\cref{lem:tech}.]
%\label{sec:prooflemtech}
%	\noindent {\bf Proof.} 
 First note that $\sum_{\ell\ge 0}2^{-p\ell} =\frac{1}{1-2^{-p}}$ so that $L_t(u) \le C_{p, \b,d}  u^{\frac{1}{2\b}}$.

\smallskip
\noindent $\blacktriangleright$ Case $p>d(1-\b)$. Note that $\sum_{\ell\ge 0}2^{-(p-d(1-\b))\ell} <+\infty$, it is clear that 
\[
L_t(u) \le C_{p, \b,d} \min\Big( u^{\frac{1}{2\b}}, \Big( \frac{u}{t}\Big)^{\frac 12}\Big).
\]
$\blacktriangleright$ Case $p = d(1-\b)$. One has 
\[
u^{\frac{1}{2\b}}\le \big(u/t)^{\frac 12}2^{d\ell(1-\b)} \Longleftrightarrow \ell\ge \ell_{t,u,\b,d}:= \bigg\lceil  \frac{(\log(ut^{\frac{\b}{1-\b}}))^+}{2d\b\log 2}\bigg\rceil.
\]
Hence, using that $p = d(1-\b)$, we get 
\begin{align*}
L_t(u)&\le u^{\frac{1}{2\b}}\sum_{\ell\ge \ell_{t,u,\b,d}}2^{-p\ell} +  \ell_{t,u,\b,d}\Big(\frac ut\Big)^{1/2}\\
& \le  u^{\frac{1}{2\b}}\frac{2^{-p \ell_{t,u,\b,d}}}{1-2^{-p}}+  \ell_{t,u,\b,d}\Big(\frac ut\Big)^{1/2}.
\end{align*}

 -- If $u \,t^{\frac{\b}{1-\b}}<1$, then $\ell_{t,u,\b,d}=0$ and  $L_t(u) \le \frac{ u^{\frac{1}{2\b}}}{1-2^{-p}}$. (Note that, under this  condition, $u^{\frac{1}{2\b}} \le (u/t)^{1/2}$).
 
 -- If $u\, t^{\frac{\b}{1-\b}}\ge 1$, then $\ell_{t,u,\b,d}\ge  \frac{\log(ut^{\frac{\b}{1-\b}})}{2d\b\log 2}$ so that
 \[
  u^{\frac{1}{2\b}}2^{-p \ell_{t,u,\b,d}} \le   u^{\frac{1}{2\b}}e^{- \frac{\log(u t^{\frac{\b}{1-\b}})}{2d\b\log 2}p} = \Big(\frac{u}{t}\Big)^{1/2}.
 \]
 Consequently, there exists a constant $C_{p,\b,d}>0$ such that 
 %\textcolor{red}{@G: ne correspond pas \`a l'\'enonc\'e}
 \[
 L_t(u) \le C_{p,\b,d} \Big(\frac{u}{t}\Big)^{1/2}\big(1+\log(ut^{\frac{\b}{1-\b}})\big).
 \]
 
 $\blacktriangleright$ Case $p < d(1-\b)$.  We consider the same $\ell_{t,u,\b}$ as in the former case and we get 
	\[
	L_t(u)\le C_{p,\b,d}\Big( u^{\frac{1}{2\b}}2^{-p\ell_{t,u,\b}}+\Big(\frac ut\Big)^{\frac 12}2^{(d(1-\b)-p)\ell_{t,u,\b,d}}\mbox{\bf 1}_{\{\ell_{t,u,\b}\ge 1\}}\Big)
	\]
	for some real constant $C_{p,\b,d}>0$.
	
	First assume that $ut^{\frac{\b}{1-\b}}> 1$. One checks that
	\[
	\frac{\log(ut^{\frac{\b}{1-\b}})}{2\b d\log 2}\le \ell_{t,u,\b,d}\le 1+\frac{\log(1+ut^{\frac{\b}{1-\b}})}{2d\b \log 2}
	\]
	where we used that $\lceil x\rceil<x+1$ for the right inequality. 
	
	 Hence
	\[
	2^{-p\ell_{t,u,\b}}\le \big(ut^{\frac{\b}{1-\b}}\big)^{-\frac{p}{2d\b}}
	%\le \big(\tfrac 12+\tfrac12 ut^{\frac{\b}{1-\b}}\big)^{-\frac{p}{2d\b}}\le c_{p,d,\b}\big(1+ ut^{\frac{\b}{1-\b}}\big)^{-\frac{p}{2d\b}}
	\]
	and 
	\[
	2^{d(1-\b)-p)\ell_{t,u,\b,d}}\le c_{p,\b,d}(1+ut^{\frac{\b}{1-\b}})^{\frac{(1-\b)d-p}{2d \b}}\le  c'_{p,\b,d}(ut^{\frac{\b}{1-\b}})^{\frac{(1-\b)d-p}{2d \b}}
	\]
	since $(1+x)^a\le 2^a x^a$ for $x\!\in[1,+\infty)$ when $a\ge 0$. Then, elementary computations yield
	\[
	\Big(\frac ut\Big)^{\frac 12}2^{d(1-\b)-p)\ell_{t,u,\b,d}}\le c'_{p,\b,d}u^{\frac {1}{2\b}} (ut^{\frac{\b}{1-\b}})^{-\frac{p}{2d\b}}
	\]
	leading to the upper-bound
	\[
	L_t(u) \le C_{p,\b,d}u^{\frac{1}{2\b}}\min\Big(1 ,\big(ut^{\frac{\b}{1-\b}}\big)^{-\frac{p}{2d\b}} \Big).
	\]
	When $ut^{\frac{\b}{1-\b}}\le1$, $\ell_{t,u,\b}=0$ so that  the above bound still holds true.
	%\hfill$\Box$
\end{proof}
	%The proof is postponed to after that of~\cref{thm:abstrait} hereafter.
	
%\noindent {\bf Remark.}	Note that, if $u^{\frac{1}{2\b}} \le \Big(\frac ut\big)^{\frac 12} $, then $ut^{\frac{\b}{1-\b}}\le 1$ so that, in the critical case 
%	$p=d(1-\b$ we couldve written ourienqyality
%	\[
%	L_t(u) \le C_{p,\b,d} \min\Big(u^{\frac{1}{2\b}}, \Big(\frac ut\Big)^{\frac 12} \Big) .
%	\]
\section{{Proof of general criterions of \cref{sec:genecrite}}}
\label{sec:4bis}
\subsection{Proof of Proposition~\ref{prop:Poincare} (Poincar\'e setting)}
\label{sec:Poincarre}

%\tcr{@fabien: chgt d'ordre qui prend en compte le changement d'ordre des r\'esultats eux-m\^emes.}
Let $f_A= \mbox{\bf 1}_A\!\in {\mathbb L}^2(\nu)$. Note that $\E_\nu f_A=Ê\nu(A)$ and ${\rm Var}_\nu(f_A) = \nu(A)(1-\nu(A))$. We start from~\eqref{eq:repere} in the preceding proof of Theorem~\ref{thm:possiblynonmarkov}, namely
\begin{align}
\nonumber \E |\nu_t(A)-\nu(A)|^2&= \frac{2}{t^2}\int_{\{0\le s\le u\le t\}}\E\big[\E[ f_A(X_u)-\nu(f_A)| {\cal F}_s](f_A(X_s)-\nu(A))\big] du \,ds\\
\nonumber						    & =  \frac{2}{t^2}\int_{\{0\le s\le u\le t\}} \E \big(P_{u-s}f_A(X_s)-\nu(A)\big)(f_A(X_s)-\nu(A))duds\\
 \label{eq:majoPoincare}& \le   \frac{2}{t^2}\int_0^t\int_{s}^t \|P_{u-s}f_A(X_s)-\nu(A)\|_2\| f_A(X_s)-\nu(f_A)\|_2   duds,
\end{align}
owing to Cauchy-Schwarz inequality. It follows from   assumption~$\ADEUX$  that
\[
\|P_{u-s}f_A(X_s)-\nu(A)\|^2= {\rm Var}_{\nu}(P_{u-s}f_{_A})_2\le e(u-s){\rm Var}_{\nu}(f_{_A}).
\]
On the other hand, one has
\[
 \| f_A(X_s)-\nu(f_A)\|^2_2 =  \nu(A)(1-\nu(A))= {\rm Var}_\nu(f_A) .
\]
Inserting the above two bounds in~\eqref{eq:majoPoincare} yields
\begin{align*}
 \E |\nu_t(A)-\nu(A)|^2 &\le \frac{2\nu_t}{t^2}\int_{\{0\le s\le u\le t\}}e(u-s)   du\,ds\\
 				& =  \frac{2{\rm Var}_{\nu}(f_A)}{t^2}\int_{0}^t \int_0^{t-s} e(v)dv \, ds\\
				  & \le \frac{2{\rm Var}_{\nu}(f_A)}{t}\int_0^t e(v)dv 
\end{align*}
so that 
\[
\|\nu_t(A)-\nu(A)\|_1\le \|\nu_t(A)-\nu(A)\|_2 \le C_{e}\frac{\sqrt{\nu(A)}}{\sqrt{t}}.
\]
On the other hand it is clear that 
\begin{align*}
\|\nu_t(A)-\nu(A)\|_1&\le  \Big\| \frac 1t \int_0^t \mbox{\bf 1}_{_A}(X_s)ds\Big|_1+ \nu(A)\\
&\le \frac 1t \int_0^t \E_{\nu} \mbox{\bf 1}_{_A}(X_s)ds+ \nu(A) =\frac 1t \int_0^t \nu(A)ds+ \nu(A) = 2 \nu(A).
\end{align*}
Hence, for every $t_0>0$,  $\AUN$\emph{(i)}  is satisfied with $\pi_t=\nu$ and $\beta = \frac 12$. Hence the bounds for Wasserstein distance follow from~\cref{thm:abstrait} with this value for $\b$.\hfill$\Box$

\subsection{{Proof of Theorem~\ref{thm:possiblynonmarkov}}}\label{subsec:proofthm2.5}Let us first deal with the $q$-moment of $\nu$. For every $M>0$, 
\[
\nu(|\cdot|^q) = \lim_{M\to+\infty} \nu(|\cdot|^q\wedge M)
\]
by Beppo Levi's monotone convergence theorem. It follows from $\ATROIS(i)$ applied with $s=0$ that 
\[
\nu(|\cdot|^q\wedge M) \le \E\big(|X_t|^q\wedge M\,|\, {\cal F}_0\big) + M\Upsilon_{t,0} e(t).
\]
Taking expectation then yields
\begin{align*}
\nu(|\cdot|^q\wedge M) &\le \E\big(|X_t|^q\wedge M\big) + M\E\,\Upsilon_{t,0} e(t)\\
&\le  \E |X_t|^q  + M\E\,\Upsilon_{t,0} e(t).
%&\le C+ M\E\,\Upsilon_{t,0} e(t)
\end{align*}
Averaging in time over $[0,t]$ the above inequality yields for any $t\ge t_0$:
\[
\nu(|\cdot|^q\wedge M) \le {C_{2,q}}+ M\E\,\Upsilon_{t,0} \frac 1t \int_0^t e(s)ds
%\quad\textnormal{ with }\quad  C_{t_0,q}=\sup_{t\ge t_0} \E |X_t|^q.
\]
where {${C_{2,q}}$ is defined in~\eqref{eq:allevatrois}}. {Now it follows from  $\ATROIS(ii)$, that  $\frac{1}{t} \int_0^t e(s)ds = O(t^{-2\gamma})\to 0$ as    $t\to+\infty$ so that: 
$$
\forall M>0,\quad  \nu(|\cdot|^q\wedge M) \le {C_{2,q}}<+\infty.
$$
 Combined with the first inequality, this proves that $\nu(|\cdot|^q)\le  {C_{2,q}}$ and 
 $${\Mpiq= 1\vee \sup_{t\ge0} M_{\pi_t}(q)}\le   {C_{2,q}}\quad \textnormal{where }\pi_t=\frac{1}{2}(\bar{\nu_t}+\nu).$$}

%\noindent \tcr{@Fabien : \`A partir de l\`a, d\'ebute la preuve  + ou - ``alternative" : c'est la  m\^eme qu'avant en intervertissant juste $q$ et $\frac{q}{q-1}$. Comme cela j'ai l'impression que plus $q$ est grand et plus $\beta$ est proche de $1/2$ ce qui me semble la bonne option (que j'ai pas toujours adopt\'ee en premi\`ere intention, j'avoue).}

\noindent Now,
\begin{equation}\label{eq:constanteexplca1}
\E|\nu_t(A)-\nu(A)|\le \nu_t(A)+\nu(A)=2\pi_t(A).
\end{equation}
On the other hand, setting $f_A= {\bf 1}_A$ and using $\ATROIS$\emph{(i)}, we have
\begin{align}
\nonumber \E |\nu_t(A)-\nu(A)|^2&= \frac{2}{t^2}\int_{\{0\le s\le u\le t\}}\E[ f_A(X_u)-\nu(f_A)(f_A(X_s)-\nu(f_A))] du \,ds\\
	\label{eq:repere}			    &= \frac{2}{t^2}\int_{\{0\le s\le u\le t\}}\E\big[\E[ f_A(X_u)-\nu(f_A)| {\cal F}_s](f_A(X_s)-\nu(f_A))\big] du \,ds\\
\nonumber & \le \frac{2}{t^2}\int_0^t\int_s^t \ee(u-s) \E[ \Upsilon_{u,s} |1_A(X_s)-\nu(A)|] du \,ds.
\end{align}
By H\" older inequality and Fubini--Tonelli's theorem, it follows that
\begin{align*}
\E |\nu_t(A)-\nu(A)|^2&\le \frac{2}{t^2}\int_0^t \int_s^t \ee(u-s) \|\Upsilon_{u,s}\|_q \,\E[1_A(X_s)-\nu(A)|^{\frac{q-1}{q}}]^{1-\frac{1}{q}} du \,ds\\
&\le \frac{2}{t}\left(\sup_{0\le s\le t} \int_0^{t} \ee(v) \|\Upsilon_{(v+s)\wedge t,s}\|_{q} dv \right)\frac{1}{t} \int_0^t\E[|1_A(X_s)-\nu(A)|^{\frac{q-1}{q}}]^{1-\frac{1}{q}}ds\\
&\le \frac{2C_{1,q}}{t^{2\gamma}}\left(\frac{1}{t} \int_0^t\E[|1_A(X_s)-\nu(A)|^{\frac{q}{q-1}}]ds\right)^{1-\frac{1}{q}},
%\left({\nu_t(A)^\frac{1}{q}+\nu(A)\right)
\end{align*}
where in the last line we used Jensen's inequality and the constant $C_{1,q}$ introduced in~\eqref{eq:allevatrois} (see~\cref{rq:allevatrois}). This constant is clearly  finite under $\ATROIS$\emph{(ii)-(iii)}. Note that 
\begin{align}\label{eq:decompA}
\nonumber \E[1_A(X_s)-\nu(A)|^{\frac{q}{q-1}}] &=  \E[1_A(X_s)](1-\nu(A))^{\frac{q}{q-1}}+\nu(A)^{\frac{q}{q-1}} \E[1_{A^c}(X_s)]\\
&\le \E[1_A(X_s)]+\nu(A)
\end{align}
since $\frac{q}{q-1}>1$. Hence
\begin{equation}\label{eq:constanteexplca2}
\frac 1t\int_0^t \E[|1_A(X_s)-\nu(A)|^{\frac{q}{q-1}}]ds \le 2\, \pi_t(A)
%= \frac{4 C_{1,q} \pi_t(A)}{t}.
\end{equation}
so that, finally
\[
\big\| \nu_t(A)-\nu(A) \big\|_2\le \frac{2^{1-\frac{1}{2q}}}{{t^{\gamma}}}\sqrt{C_{1,q}}\,\pi_t(A)^{\frac 12(1-\frac 1q)}.
\]
Combining~\eqref{eq:constanteexplca1} and~\eqref{eq:constanteexplca2} and noting that $\|\cdot\|_1\le \|\cdot\|_2$, we conclude that $\AUN$\emph{(i)} is satisfied with $K=2 (1\vee \sqrt{C_{1,q}})$. As concerns $\AUN$\emph{(ii)}, we remark that,
$$
 M_{\pi_t}(q)=\frac{1}{2}\left(\frac{1}{t}\int_0^t \E[|X_s|^q] ds+\nu(|\,.\,|^q)\right)\le \frac{1}{2}\left(C_{2,q}+\nu(|\,.\,|^q)\right),$$
where $C_{2,q}$ has been introduced  in~\eqref{eq:allevatrois} (see~\cref{rq:allevatrois}). But since $\AUN$\emph{(i)} is satisfied, $(\nu_t)$ weakly converges to $\nu$ so that by a classical argument, 
$$
\nu(|\,.\,|^q)\le \limsup_{t\rightarrow+\infty}  \frac{1}{t}\int_0^t \E[|X_s|^q] ds\le C_{2,q}
$$
and hence $\sup_{t\ge0} M_{\pi_t}(q)\le C_{2,q}<+\infty$ under $\ATROIS$\emph{(iii)}. Hence $\AUN$ holds  for every $\tzero>0$ as announced. The bounds in~\eqref{eq:boundsabstrait} then straightforwardly follow from~Theorem~\ref{thm:abstrait} for this value of $\b$.
%\end{proof}

\subsection{Proofs of   Propositions~\ref{prop:Markovsetting}, \ref{prop:markovcontraction} and Corollary~\ref{cor:Contraction}}\label{sec:proof4.2}
%Let $(\F_t)_{t\ge 0}$ denote the  natural filtration of $(X_t)_{t\ge 0}$.

\noindent {\bf Proof of Proposition~\ref{prop:Markovsetting}.} Let $({\cal F}_t)_{t\ge 0}$ denote the augmented natural filtration of $(X_t)_{t\ge 0}$ so that $(X_t)_{t\ge 0}$ is a homogeneous $({\cal F}_t)$-Markov process with semi-group $(P_t)_{t\ge 0}$. Let us prove briefly that $\nu $ is the unique invariant distribution of $(P_t)$. Let $f:\R^d\to [0,1]$ be a bounded Borel function.  Then, it follows from~\eqref{eq:TVboundby}$(i)$ that
%It is clear that $P_t(f)\to \nu(f)$ for every such $f$ .
$$\Big| \frac 1t\int_0^t \!\! P_s(f)(x)ds-\nu(f)\Big|\le \frac{\|f\|_{\sup}}{t}\int_0^t\! \|P_s(x,dy)-\nu\|_{TV} ds\le  \frac{\|f\|_{\sup}}{t}\psi(x)\int_0^{+\infty}\hskip-0.25cm e(s)ds\; \mbox{as $t\to+\infty$}.
$$ 

Hence, one classically derives that $\nu$ is invariant and  if $\nu'$ is also an invariant distribution then, still for any Bounded Borel function $f$,  $\nu'(f) =  \nu'P_t(f)(x) = \int P_t(f)(\xi)\nu'(d\xi)  \to \nu(f)$ so that $\nu' = \nu$.

\smallskip
Let $g:\R^d\to [0,1]$ be a bounded Borel function. One has for every $s$, $t\!\in [0,+\infty)$, $t\ge s$,
\begin{align*}
|\E\,(g(X_t)\,|\, {\cal F}_s)-\nu(g)| = |P_{t-s}g(X_s)-\nu(g)|\le \psi(X_s)e(t-s)
\end{align*}
so that we may set $\Upsilon_{t,s}= \psi(X_s)$ to fulfill~$\ATROIS(i)$. Condition~$\ATROIS(iii)$ straightforwardly  follows from~\eqref{eq:TVboundby}$(i)$ since $\|\psi(X_t)\|_q= \big(\mu_0P_t(\psi^q)\big)^{\frac 1q}$ and $\E\,|X_t|^q= \mu_0P_t |\cdot|^q$. \hfill$\Box$
%\[
%C_{1,q}= \int_0^{+\infty}\big(\mu_0P_s(\psi^q)\big)^{\frac 1q}e(s)ds.\hskip 9cm\Box
%\]

\bigskip
\noindent {\bf Proof of Proposition~\ref{prop:markovcontraction}.} We check the assumptions of Proposition~\ref{prop:Markovsetting}. Let $A$ be a Borel set of $\R^d$. By assumption $(i)$ we know that $f_A(x) = \E_x {\bf 1}_A(X_{\theta_0})= P_{t_0}(\mbox{\bf 1}_{_A})(x)$ is Lipschitz continuous with $[f_A]_{\rm Lip}\le c(\theta_0)$. Now let $t\ge \theta_0$. One has $P_t\mbox{\bf 1}_A= P_{t-\theta_0}f_A$. Consequently
\[
|P_t\mbox{\bf 1}_A(x)-\nu(A)|\le |\mu_0P_{t-\theta_0}f_A(x)-\nu (f_A)|\le c(t_0) {\cal W}_1(\mu_0P_{t-\theta_0}(x,dy),\nu)\le c(t_0)\psi(x) e(t-\theta_0).
\]
Consequently, setting $\tilde e(t):= e(t-\theta_0)$, $t\ge \theta_0$, one has for every $t\ge \theta_0$,
\[
\big\|P_t(x,dy)-\nu\big\|_{TV} \le \psi(x) \tilde e(t).
\]
On the other hand, it is clear that, when $t\!\in [0,\theta_0]$, $\big\|P_t(x,dy)-\nu\big\|_{TV} \le 2\le  2\big(1\vee\psi(x)\big)$. Set $\tilde e (t)= e(t-\theta_0) $, $t\!\in [0,\theta_0]$, and one straightforwardly checks  that $\tilde e \!\in \mathbb{L}^1([0, +\infty))$ so that $\ATROISbis(i)$ is fulfilled. Finally, $\ATROISbis(ii)$ follows as in the above proof. \hfill$\Box$
%\[
%\int_0^{+\infty}\big(\mu_0P_t(\psi\vee 2)^q)^{\frac 1q}\tilde e(t)dt <+\infty.
%\]
%This completes the proof.\hfill $\Box$

%\smallskip \tcr{@F \&G sous la condition actuelle en TV, la condition sous forme s\'erie brute est mal commode ici; 
%Il faudrait pouvoir le commencer en $t_0$ ou qq chose comme \c ca mais on doit pouvoir bidouiller. M\^eme si en fait on s'en fout.}

\medskip
\noindent {\bf Proof of~\cref{cor:Contraction}.} Let us consider a  test Lipschitz continuous function  with $[f]_{\rm Lip}\le 1$,
\begin{align*}
P_tf(x) -\nu(f) = P_tf(x) -\nu P_tf&=   P_tf(x) -\int_{\R^d}\nu(dx')P_tf(x')\\
			      	     & =  \int_{\R^d} \big(P_tf(x) - P_tf(x')\big)\nu(dx')\\
				      &\le \int_{\R^d} \W_1(P_t(x,dy),P_t(x',dy))\nu(dx')\\
				      &\le \int_{\R^d} \Psi(x,x')\nu(dx')\ee(t):=\psi(x)e(t),
\end{align*}
where we used the Monge--Kantorovich representation of $\W_1$-distance in the third line and $\ATROISter(i)$ in the last one. Taking the supremum over test-functions $f$ yields the announced result.\hfill$\Box$ 

\section{Applications to Brownian diffusions}
\label{sec:5bis}
\subsection{Proof of~\cref{prop:weakreverting}} \label{subsec:proof2.12} We apply~\cref{prop:Markovsetting}. To this end, we appeal to~\cite[Theorem 3.2]{doucfortguillin}. Without loss of generality, we can assume that $a\in(0,1)$. Setting ${\cal V}(x)=1+|x|^2$ and for a given $p\ge 1$,  $V_p={\cal V}^p$  one checks that  for some positive $\tilde{\beta}$ and $\tilde{\alpha}$,  ${\cal L}V_p\le \tilde{\beta}-\tilde{\alpha} {\cal V}^{p+a-1}=\tilde{\beta}-\tilde{\alpha} V_p^{1+\frac{a-1}{p}}$. This implies that Condition $(ii)$ of~\cite[Theorem 3.2]{doucfortguillin} holds with $\phi(s)=s^{1+\frac{a-1}{p}}$, $s\ge 1$. Furthermore, setting 
$H_\phi(s)=\int_1^s \frac{du}{\phi(u)}$, one can check that $r_*=\phi\circ H_\phi^{-1}$ satisfies $r_*(s)\sim s^{\frac{p}{1-a}-1}$ as $s\rightarrow+\infty$. Owing to the ellipticity condition which ensures the irreducibility condition $(i)$, we deduce from~\cite[Theorem 3.2]{doucfortguillin} (Eq. $(3.5)$) applied with $\Psi_1={\rm Id}$ and $\Psi_2=1$ that
$$
\|P_t(x,dy)-\nu\|_{TV}\le C \psi(x)  \ee(t),
$$
with $\ee(t)=1\wedge t^{1-\frac{p}{1-a}}$ and $\psi(x)=1+|x|^{2(p+a-1)}$. Thus, taking $p$ large enough, $\ee\in \mathbb{L}^1([0,+\infty)$ so that Condition $(i)$ of~\cref{prop:Markovsetting} holds. Owing to~\eqref{eq:condaltmark}, it is now enough to check that for any $q>0$
\begin{equation}\label{eq:condaltmarltoch}
\int_0^{+\infty}  (\E_{\mu_0}|X_t|^{q}) \ee(t) dt+\sup_{t> 0} \frac{1}{t}\int_0^t  \E_{\mu_0}|X_s|^{q}ds<+\infty.
\end{equation}
To this end, we deduce from the inequality ${\cal L}{\cal V}^p\le \tilde{\beta}-\tilde{\alpha} {\cal V}^{p+{a-1}}$ that 
$$ \frac{1}{t}\int_0^t \E_{\mu_0}[{\cal V}^{p+{a-1}}(X_s)] ds\le \mu_0 ({\cal V}^p)<+\infty$$
and from the Itô formula applied to $F(t,X_t)=\ee(t) {\cal V}^p(X_t)$ that $(S_t)_{t\ge1}$ defined by
$$ S_t:=\ee(t) {\cal V}^p(X_t)+\tilde{\alpha} \int_0^t {\cal V}^{p+{a-1}}(X_s) \ee(s) ds+\tilde{\beta} \int_t^{+\infty}\ee(s) ds,$$
is a non-negative super-martingale. This implies that $\sup_{t\ge1} \E_{\mu_0}[S_t]<+\infty$ so that 
$$\int_0^{+\infty} \E_{\mu_0}[{\cal V}^{p+{a-1}}(X_s)] \ee(s) ds <+\infty.$$
Using the fact that the above properties hold for any $p>0$,~\eqref{eq:condaltmarltoch} easily follows for any $q>0$ (taking $p$ such that $q=\frac{p+a-1}{2}$).
\begin{rem}
Note that we only proved polynomial rate of convergence to equilibrium since it is enough to prove our results. It is worth noting that~\cite{doucfortguillin} allows to get sub-exponential rates (but this requires to assume that $\mu_0$ has corresponding sub-exponential moments).
\end{rem}%\subs
\subsection{{Proof of \cref{thm:diffusionsconfluentes} and \cref{cor:A tous moments}}}\label{subsec:proof2.10-2.11}
\paragraph{{Proof of Theorem~\ref{thm:diffusionsconfluentes}.}} {$(\alpha)$}  Let $q\ge 2$. The existence of an invariant distribution, always   lying in ${\cal P}_q(\R^d)$, follows from  Hajek's criterion $({\rm \bf Haj)}_q$ applied to the Lyapunov function $|\cdot|^q$ since elementary computations based on It\^o's formula prove that ${\cal L}|\cdot|^q \le \b_q-\a_q |\cdot|^q$ for some $\alpha_q>0$ which in turn classically yields (see e.g.~\cite{EthKur}) the existence of an invariant distribution lying in ${\cal P}_q(\R^d)$. Uniqueness straightforwardly follows from the confluence properties~\eqref{eq:Lq-confluence} or~\eqref{eq:Lq-confluenceelliptic} by setting $\mu_1=\nu$ and $\mu_2=\nu'$ where $\nu$ and $\nu'$ are both invariant and lie in $L^r(\R^d)$ or $L^1(\R^d)$, hence always  in $L^1(\R^d)$. Moreover, see again~\cite{EthKur}, we have $P_t(|\cdot|^q)(x)\le e^{-\a_q t}|x|^q + \frac{\b_q}{\a_q}$ so that, if $\mu_0\!\in {\cal P}_q(\R^d)$, then $\mu_0P_t(|\cdot|^q) \le \mu_0(|\cdot|^q) +\frac{\b_q}{\a_q}$ for every $t\ge 0$.
 
 %\noindent \tcr{{\bf @Fabien: Avec la	nouvelle fct. Lyap. \c ca devient inutile ? On dit juste que c'est pareil avec $(1+|x|^2)^{q/2}$}. Assume now $q\!\in [1,2)$. Then $|\cdot|^q$ is $C^2$ only on $\R^d\setminus\{0\}$. We define $\varphi_{\ve_0}: [0,+\infty)\to [0, +\infty)$  a $C^2$ function such that  $\varphi''_{\ve_0}$ is a probability density supported by $[\ve_0/2,\ve_0]$, $\ve_0 >0$ being fixed, and $\varphi_{\ve_0} (0)= \varphi'_{\ve_0}(0)=0$. Note that $u= \varphi_{\ve}(u)+\ve -\varphi_{\ve}(\ve)$, $u\ge \ve$. Then set $V_{q,\ve_0}(x) = \varphi_{\ve_0}(|x|^q)$. The function $V_{q,\ve_0}$ is $C^2$ on $\R^d$ and one checks that $V_{q,\ve_0}$ is coercive, ${\cal L}(V_{q,\ve_0})\le \b_{q,\ve_0}-\a_{q,\ve_0}V_{q,\ve_0}$  so that an invariant distribution $\nu$ exists since such a diffusion is a Feller Markov process and any such invariant distribution satisfies $\nu(V_{q,\ve_0})<+\infty$ which clearly implies  $\nu(|\cdot|^q)<+\infty$ and $\sup_t\mu_0P_t(|\cdot|^q)<+\infty$.}
%
 {When $q\!\in[1,2)$, one proceeds likewise with a ${\cal C}^2$-function $V_q$ which satisfies $V_q(x)= |x|^{ q}$ on $B(0,1)^c$. For such a function, the fact that $\sup_{x\in \bar{B}(0,1)} {\cal L} V_q(x)<+\infty$ combined with ${\cal L}|\cdot|^q \le \b_q-\a_q |\cdot|^q$, for some $\alpha_q>0$ on $B(0,1)^c$ allows to obtain that ${\cal L}V_q \le \b'_q-\a'_q V_q$ with $\alpha'_q>0$. The sequel is then very similar to what precedes.}

 On the other hand, if we set $\mu_1= \delta_x$ and $\mu_2=\delta_y$  so that ${\cal W}_q(\mu_1,\mu_2)= |x-y|$ in~\eqref{eq:Lq-contraction} and~\eqref{eq:L1-contractionelliptic}. We can conclude that $\ATROISter$ holds true with $\psi$ and $\ee$ as above. Hence $\ATROISbis$ from Proposition~\ref{prop:markovcontraction}$(ii)$ is fulfilled in turn.
 
\noindent  {$(\beta)$}   It remains to prove  the condition~{$\LSF$} (see~\eqref{eq:LipStrgFeller}) of Proposition~\ref{prop:markovcontraction}. For that purpose we rely on the Bismuth--Elworthy--Li formula (see e.g.~\cite{GilPag2026} in a 1D setting): for every 
%\begin{prop}  \label{prop:ellipticity} Assume $d'=d$ and $\sigma$ is uniformly elliptic in the following sense
%\begin{equation}\label{eq:sigmaUellip}
%\exists\, \underline \sigma_0>0 \mbox{ such that } \forall\, x,\, \xi \!\in \R^d,Ê\quad  \xi^{\top} \sigma^{\top}\sigma (x)\xi = |\s(x)\xi|^2 \ge \underline \sigma_0^2 |\xi|^2.
%\end{equation}
bounded Borel function $f:\R^d\to \R$, $\E\, f(X^x_t)$ is differentiable as a function of $x$ and for every $t>0$, 
\begin{equation*}
%s\label{eq:BEL}
\nabla_x \E\, f(X^x_t)= \frac 1t\E\Big[ f(X^x_t) \int_0^{t}   \big\langle \s^{-1}(X^x_s)Y^{(x)}_s,dW_s\big\rangle\Big].
\end{equation*}
%where $\sigma$ is uniformly elliptic in the sense that there exists $\ve_0>0$ such that
%\begin{equation}\label{eq:sigmaUellip}
%\forall\, \xi \!\in \R^d, \quad \xi^{\top}\sigma\sigma^{\top}(x)\xi \ge \ve_0|\xi|^2.
%\end{equation}
where $(Y^{(x)}_t)_{t\ge 0}$ denotes the tangent process of $tX^x_t)$.
Then one has by It\^o's isometry and Fubini--Tonelli's Theorem
%\end{prop}
%
%\noindent {\bf Proof.}  Formula~\eqref{eq:BEL} is known as Bismuth-Elworthy-Li formula (see~\cite{GilPag2018} in a 1D setting).  As for the bound, one has
\[
|\nabla_x \E\, f(X^x_t)|^2\le \|f\|_{\sup}^2\bigg\| \int_0^{t}   \big\langle \s^{-1}(X^x_s)Y^{(x)}_s,dW_s\rangle \bigg\|^2_2=  \|f\|_{\sup}^2\bigg[\int_0^t\E\,|\s^{-1}(X^x_s)Y^{(x)}_s|^2 \bigg].
\]
Now
\begin{align*}
\E\,|\s^{-1}(X^x_s)Y^{(x)}_s|^2&= \E \,Y_s^{(x),*}(\sigma(X^x_s)^){-1}(\sigma(X^x_s))^{-1})^{\top}Y_s^{(x)}\\
&= \E\, (Y_s^{(x)})^{\top} (\sigma\sigma^{\top}(X^x_s))^{-1}Y_s^{(x)}\le \underline \sigma_0^{-2} \E \,|Y_s^{(x)}|^2
\end{align*}
from which one easily derives that
%Moreover, for every $  t>0$ 
%\textcolor{red}{@Gil  : ref. sur le flot tangent $\mathbb{L}^2$-born\'e}
\[
\sup_{x\in \R^d} |\nabla_x \,\E\, f(X^x_t)| \le C(t)\|f\|_{\infty} \quad \mbox{ with }\quad C(t)= \frac{1}{\underline \sigma_0\sqrt{t}}\sup_{0\le s\le t, \, x \in \R^d} \|Y^{(x)}_s\|_{_2}<+\infty
\]
so that {$ x\mapsto P_tf(x)= \E\, f(X^x_t) $ is Lipschitz continuous and $[P_tf]_{\rm lip} \le C(t)\|f\|_{\infty}$}. 
% the announced result.\hfill$\Box$
 
\smallskip 
\noindent  {$(\gamma)$}  It follows from the previous steps that the conclusions of Proposition~\ref{prop:markovcontraction}  hold true, hence those of Proposition~\ref{prop:Markovsetting}  so that, finally, Theorem~\ref{thm:abstrait} applies for any $t_0>0$ with  $\pi_t= \frac 12 (\bar \nu_t +\nu)$, $\b= \frac 12(1-\frac 1q)$ for $0<p<q$.   \hfill$\Box$
}

%\bigskip
%\noindent {\bf Remarks.} $\bullet$ Uniqueness of the invariant distribution $\nu$ is ensured by any of the two $L^1$-confluence and uniform ellipticity assumptions. \textcolor{red}{@G \&F: regularit\'e des coeffs dans l'uniforme ellipticit\'e quand m\^eme et bornitude de $\sigma$ ? }.
%
%\smallskip
%\noindent $\bullet$ Note that the condition  $q>d_+$ in the second and fourth  case of~\eqref{eq:boundsWasserdiff}
%is  just a necessary constraint on $q$ for the condition on $p$ to possibly  hold. Otherwise the case is empty.
% and if $p>\frac{d-1+\sqrt{(d+1)^2+4d}{4}$ then 

\paragraph{Proof of Corollary~\ref{cor:A tous moments}.}
%\tcr{@Fabien \& Gil: en fait il faudrait prendre $\varphi(|\cdot|^q)+1$ avec  $\varphi$ nulle au voisinage de $0$, croissante et \'egale \`a l'identit\'e au-del\`a de $1$ par exemple mais bon je craque un peu.}
%
Let $r\!\in [1,2]$ and $V_r(x)= (1+|x|^2)^{\frac r2}$. As a preliminary note that
\[
|x|^r \le V_r(x)\le 1+|x|^r
\]
since $u^{\frac r2}$ is Holder. One has, for every $x\!\in \R$, 
$$
 \nabla V_r(x)= r(1+|x|^2)^{\frac r2-1}x \quad \mbox{ and }\quad \nabla^2V_r(x)= r(1+|x|^2)^{\frac r2-1}\Big(I_d + (r-2)\tfrac{x}{(1+|x|^2)^{\frac 12}}\big(\tfrac{x}{(1+|x|^2)^{\frac 12}}\big)^{\top}\big).
$$

Then {one checks that}
{\begin{align*}
{\cal L}V_r(x)&= \langle \nabla V_r(x)\,|\, b(x)\rangle +\tfrac 12 {\rm Tr}\big(\sigma^{\top}(x)D^2V_r(x)\sigma(x)\big)\\
&= r(1+|x|^2)^{\frac r2-1} \Big(  \langle b(x)\,|\, x\rangle + \tfrac12 \|\sigma(x)\|^2_{_F}+ (\tfrac r2-1) \frac{|(\s(x)^{\top}x|^2}{1+|x|^2}\Big)\\
&\le \bar \kappa'-\underline \kappa' V_r(x)
\end{align*}
}
%with $\a= r\bar \kappa$ and $\b = r\underline \kappa$ 
\noindent {for some $\underline \kappa'>0$ owing to \eqref{eq:corcontsbn}. Let $\lambda>0$ . Then
\begin{equation}\label{eq:expLyap}
{\cal L}e^{\lambda V_r} \le  \lambda e^{\lambda V_r}\big({\cal L}V_r+\tfrac{\lambda}{2}  {|\sigma^{\top} \nabla V_r|^2}\big).
\end{equation}
Note that 
\begin{align*}
|\s(x)^{\top}\nabla V_r(x)|^2 \le \|\s(x)\|_{_F}^2|\nabla V_r(x)|^2&\le C_\s^2 (1+|x|)^{2-r}r^2(1+|x|^2)^{ r-2}|x|^2\\
&\le C'_{\s,r} (1+|x|^2)^{\frac r2}= C'_{\s,r}V_r(x).
%r^2C_\s^2(1+|x|)^{r}\le  2^{r-1}r^2C_\s^2V_r(x).
\end{align*}
 Choose $\lambda \!\in (0,\frac{2\underline \kappa'}{C'_{\s,r}})$. Set $\tilde \alpha = \underline \kappa'-C'_{\s,r}\tfrac{\lambda}{2}>0$ and $v= \frac 12 +\frac{ \bar \kappa'}{\tilde \alpha}$.
  Inserting the above  inequality in~\eqref{eq:expLyap}  yields 
\begin{align*}				  
 {\cal L} e^{\lambda V_r} & \le   \lambda e^{\lambda V_r} \Big( {\cal L}V_r +  C'_{\s,r} \tfrac{\lambda}{2} V_r\Big)\\
				   & \le \lambda e^{\lambda V_r}\big(\bar \kappa'-\tilde \alpha V_r\big)\\
				   & = \lambda e^{\lambda V_r}\big(\bar \kappa'-\tilde \alpha V_r \big)\mbox{\bf 1}_{\{V_r\le v\}} +\lambda e^{\lambda V_r}\big(\bar \kappa'-\tilde \alpha V_r\big)\mbox{\bf 1}_{\{V_r>v\}}\\
				   & \le \lambda\bar \kappa' e^{\lambda v}\mbox{\bf 1}_{\{V_r\le v\}}  -\tfrac {\lambda \tilde \a}{2} e^{\lambda V_r} \mbox{\bf 1}_{\{V_r>v\}}\\
				   &\le  \lambda (\bar \kappa' +\tfrac{\tilde \a}{2})e^{\lambda v}-\tfrac {\lambda \tilde \a}{2} \tilde \a e^{\lambda V_r}\\
				   & = \tilde \b' -\tilde \a'e^{\lambda V_r} \quad \mbox{with }\quad \tilde \a'>0,
\end{align*}
where we used in the fourth line that $\bar \kappa'-\tilde \a V_r\le -\frac{\tilde \a}{2}$ on $\{V_r>v\}$.
Classical arguments show that then  the invariant distribution $\nu$ has a finite exponential moment of the form $\nu(e^{\lambda V_r}) <+\infty$. Moreover for any distribution $\mu_0$ such that $\mu_0\big(e^{\lambda V_r}\big)< +\infty$, one has for every $t>0$, 
\[
\mu_0P_t e^{\lambda V_r} \le e^{\tilde \a' t} \mu_0\big(e^{\lambda V_r}\big)+\frac{\tilde \b '}{\tilde \a'}
\]
so that $\displaystyle \sup_{t\ge 0} \mu_0P_t e^{\lambda V_r} <+\infty$. In particular, for every $q\ge 1$,  $\displaystyle  \sup_{t\ge 0} \mu_0P_t |\cdot|^{q}<+\infty$. As a consequence we may apply the bounds obtained in Theorem~\ref{thm:possiblynonmarkov} with as large values $q$ as needed since $\nu(|\cdot|^q)+ \sup_{t\ge 0}\bar \nu(|\cdot|^q)<+\infty$ for every $q\ge 1$. \hfill$\Box$

%ection{SDEs driven by weak memory fractional processes}\label{sec:fractionalres}

%\subsection{McKean-Vlasov (why not??)}

%\subsection{Euler schemes}%
%J'ai l'impression que le sch\'ema d'Euler à pas d\'ecroissant ne va pas assez vite vers la proba invariante mais bon...

\section{Proofs related to~\cref{sec:fracsdes} (Gaussian and fractional driven SDEs)}\label{annexproofTVfrac}\label{sec:6}
\paragraph{Proof of~\cref{thm:fractrao}} \noindent {\sc Step 1}. 
The proof is mainly an adaptation of~\cite[Proposition 6]{ANPSieber} ({which is devoted to the fractional case $g(t)=t^{H-\frac{1}{2}}$}). We use the same notations: 
%The space ${\cal H}_H$ mentioned in~\cref{sec:fractionalres} is precisely, the closure of the space of ${\cal C}^\infty$ functions from $\ER^{-}$ to $\ER^d$ in the norm
%$$\|f\|_{{\cal H}_H}:=\sup_{s,t\le 0}\frac{|f(t)-f(s)|}{|t-s|^{\frac{1-H}{2}} \sqrt{1+|t|+|s|}}.
we set $\ell(x,w)=(\ell_t(x,w))_{t\ge0}$ the deterministic function given (when {it} makes sense) by
$$\ell_t(x,w)=x+\int_{-\infty}^0 g(t-u)-g(-u) dw_u,\quad t\ge0.$$
First, note that under $\HFBM (iii)$, $|g'(v)|\lesssim (1\vee v)^{1-\zeta}$ so that for $u<0$ and $t\ge0$, $|g(t-u)-g(-u)|\lesssim t (1\vee u)^{1-\zeta}$ which implies that
$\PE_{W^{-}}(dw)$-$a.s.$, for all $t\ge0$, for $\varepsilon\in(0,\zeta)$,
$$\limsup_{u\rightarrow-\infty} |(g(t-u)-g(-u)) w_u|\lesssim\limsup_{u\rightarrow-\infty}  u^{\frac{3}{2}-\zeta+\varepsilon} \sup_{v\ge1 } v^{-\frac{1}{2}-\varepsilon} w_v=0.$$
Then, an integration by parts combined yields: $\PE_{W^{-}}(dw)$-$a.s.$,
\begin{equation}\label{IPPeq}
\ell_t(x,w)=x+\int_{-\infty}^0 g'(t-u) w_u du, \quad t\ge0.
\end{equation}
One can check that $\PE_{W^{-}}(dw)$-$a.s.$, $t\mapsto \ell_t(x,w)$ is ${\cal C}^\infty$ on $(0,+\infty)$ and locally $\alpha$-H"older on $[0, +\infty)$ for any $\alpha<H$. In the sequel, we denote $\Cloc^{H-}(\R_+,\R^d)$ the related space. With these notations, one can check (see~\cite[{Eq. (19)}]{ANPSieber} for details) that
$$ {\cal L}(X_t|{\cal F}_s)={\cal L}(\Phi_{t-s}(\ell(x, w)))_{x=X_s,w=(W_{u+s})_{u\le0}},$$
where  for a given deterministic path  $\ell\in\Cloc^{H-}(\R_+,\R^d)$,  $\Phi_t(\ell)$ denotes the unique  solution to
\begin{equation}\label{eq:innovation_sde}
  \Phi_t(\ell)=\ell(t)+\int_0^t b\big(\Phi_s(\ell)\big)\,ds+\sigma\tilde{G}_t,\quad t\geq 0.
\end{equation}
Note that existence and uniqueness easily follow from the Lipschitz assumption on $b$. By disintegrating the invariant distribution $\Pi$ (with marginal $\nu$) of the Markov process $Z$, we have for a given bounded measurable function $h:\ER^d\rightarrow[0, +\infty)$.
$$
\E[h(X_t)|{\cal F}_s]-\nu(h)=\Psi_h(X_s,(W_{u+s})_{u\le0}),
$$
where
$$
\Psi_h(x,w)=\int \E[h(\Phi_{t-s}(\ell(x, w)))-h(\Phi_{t-s}(\ell(y, \tilde{w})))]\Pi(dy,d\tilde{w}).
$$
This implies that
\begin{equation}\label{eq:tvcondpeiozp}
\hskip-1cm \|{\cal L}(X_t|{\cal F}_s)-\nu\|_{TV}\le \int \|{\cal L}(\Phi_{t-s}(\ell(x, w)))_{x=X_s,w=(W_{u+s})_{u\le0}}-{\cal L}(\Phi_{t-s}(\ell(y, \tilde{w})))\|_{TV}\Pi(dy,d\tilde{w}).
\end{equation}
This thus suggests to exhibit some bounds for
$$ \|{\cal L}(\Phi_{t}(\ell(x, w)))-{\cal L}(\Phi_{t}(\ell(y, \tilde{w})))\|_{TV},$$
for any $(x,w)$, $(y,\tilde{w})$. In other terms, the aim is to study the behavior of the dynamics {conditioned to their past before time $0$}.  To this end, the idea is to first obtain $L^1$-bounds at time $t$ (with a synchronous coupling) and then to deduce TV-bounds from a final coalescent coupling (which succeeds with high probability in view of the first part). These two parts correspond to Steps $2$ and $3$ below.

\noindent {\sc Step 2} ($L^1$-bounds). This part is an adaptation of~\cite[Corollary 2]{ANPSieber}. The two main points are the following. First, if we assume that 
$g(t)=t^{H-\frac{1}{2}}$ on $(0,1]$, Lemma 1 and Proposition 5 of~\cite{ANPSieber} are not modified. When $t_0\neq 1$, this only involves to replace   the integer subdivision in the proof of~\cite[Proposition 2]{ANPSieber} by a subdivision $\{ kt_0,k\in\mathbb{N}\}$. We choose to leave these details to the reader. Second, in the adaptation of the proof of~\cite[Corollary 1]{ANPSieber}, the main point is to deduce from~\eqref{IPPeq} that, for any $t>0$, 
$$
\dot{\ell}_t(x,w)-\dot{\ell}_t(y,\tilde{w})=\int_{\infty}^0 g''(t-u) (\tilde{w}_u-w_u) du,
$$
so that by $\HFBM (iii)$, for $\varepsilon$ small enough,
\begin{align*}
|\dot{\ell}_t(x,w)-\dot{\ell}_t(y,\tilde{w})|&\lesssim \|\tilde{w}-w\|_{\infty,[-1,0]}+\sup_{u\le -1} \frac{|\tilde{w}_u-w_u|}{u^{\frac{1}{2}+\varepsilon}} \int_{-\infty}^{-1} |t-u|^{-\zeta} |-u|^{\frac{1}{2}+\varepsilon} du\\
&\lesssim \mathfrak{C}_\varepsilon(\tilde{w}-w) \frac{t^{-\zeta+\frac{3}{2}+\varepsilon}}{\zeta-\frac{3}{2}-\varepsilon}\quad\textnormal{with}\quad \mathfrak{C}_\varepsilon(w)=\sup_{u\in\ER_{-}} \frac{|w_u|}{(1\vee u)^{\frac{1}{2}+\varepsilon}}.
\end{align*}
A careful reading of the proof of~\cite[Corollary 2]{ANPSieber} then leads to: for any (small) $\varepsilon>0$,
\begin{equation*}
      \E[|\Phi_{t}(\ell(x,w)) - \Phi_{t}(\ell(y,\tilde{w}))|]\le C_\varepsilon\left(  e^{-c t}|x-y|+ \mathfrak{C}_\varepsilon(w-\tilde{w}) t^{-\zeta+\frac{3}{2}+\varepsilon}\right),
        \end{equation*}
        where $c$ is a positive constant (independent of $\varepsilon$). 
        
\noindent {\sc Step 2.} (TV-Bounds). Here, this is an adaptation of~\cite[Proposition 6]{ANPSieber}. Once again, since we assume that $g(t)=t^{H-\frac{1}{2}}$ on $(0,t_0]$, one can check that, at the price of sticking the paths at time $t_0$ instead of $1$, the strategy still works and the result only differs by the previous Wasserstein bounds. A careful reading leads to: for any $\varepsilon>0$, there exists $C_\varepsilon>0$ such that for any $t\ge 0$, 
\begin{equation}\label{eq:TVbounddklqjq}
 \|{\cal L}(\Phi_{\tau}(\ell(x,w)) - {\cal L}(\Phi_{\tau}(\ell(y,\tilde{w})))\|_{{\rm TV}}\le C_\varepsilon \left(e^{-c t}|x-y|+{\mathfrak{C}_\varepsilon(w-\tilde{w})} {(1\vee t)}^{-\zeta+\frac{3}{2}+\varepsilon}\right).
 \end{equation}
In other words, the cost for sticking the paths does not modify the orders of convergence.

\smallskip
\noindent We are now ready to conclude the proof by plugging~\eqref{eq:TVbounddklqjq} into~\eqref{eq:tvcondpeiozp}. This yields: for any $\varepsilon>0$, there exists $C_\varepsilon>0$ such that for any $0\le s\le t $,
\begin{equation*}
\|{\cal L}(X_t|{\cal F}_s)-\nu\|_{TV}\le C_\varepsilon   \left(C_1(X_s) e^{-c (t-s)}+ C_2((W_{u+s})_{u\le 0})(1\vee(t-s)^{-\zeta+\frac{3}{2}+\varepsilon}\right),
\end{equation*}
with 
$$ 
C_1(x)=\int |x-y| \nu(dy)\quad C_2(w)= \E[{\mathfrak{C}_\varepsilon(w-W^{-})}].
$$
It is classical to check that  that $\int |y|^q\nu(dy)<+\infty$ for all $q\ge1$  and that $\sup_{t\ge0} \E[|X_t|^q]<+\infty$ as soon as $\int |y|^q \bar{\mu}(dy)<+\infty$ (see~\cite[Proposition 3.12]{Hairer2005} {for the fBm and~\cite[Proposition A.4]{panloup-richard} for the general case}). Furthermore, the fact that $\mathfrak{C}_\varepsilon(W^{-})$ has moments of any order for any $\varepsilon>0$ is also a classical property of the Brownian motion. The same property easily follows for $C_2\big((W_{u+s})_{u\le 0}\big)$ by the triangle inequality} {and leads to the announced result.
%
 %   
%Precisely, this leads to: for $\ell_1,\ell_2\in{\cal C}^1(\ER_,\ER^d)$, 
%$$\E[|\Phi_{t}(\ell_1) - \Phi_{t}(\ell_2)|^2]\leq C\left(e^{-c t}\E[|\Phi_{t_0}(\ell_1) - \Phi_{t_0}(\ell_2)|^2] +\sum_{k=1}^{\lfloor\frac{t}{t_0}\rfloor}\rho^{\frac{t}{t_0}-k}\|{\dot{\ell}_1-\dot{\ell}_2}\|_{\infty;[k,(k+1)]}^2\right).
%
% \E[|\Phi_{\tau}(\ell(x,w)\TCR{)} - \Phi_{\tau}(\ell(y,\tilde{w}))|^2]\lesssim_{\varepsilon}  e^{-c\tau}|x-y|^2+ \TCR{\frC(w-\tilde{w},\epsilon))^2}\TCR{(1\vee\tau)}^{2H-2+\varepsilon}.
%        \end{equation*}
% and is based on the following approach. Set $\tilde{t}_0=t_0\wedge 1$. Let $t\ge0$.  To get $TV$-bounds at time $t+\tillde{t}_0$,  use a synchronous coupling from  $0$ to $t$ in order to obtain Wasserstein bounds and on $[t,t+\tilde{t}_0]$, build a coupling which ensure that the paths stick at time $t+\tilde{t}_0$. The basic idea is that when the paths are close at time $t$, then the cost to stick the paths may be weak. Let us first recall how to obtain Wasserstein bounds at time $t$.
%First, let us remark that 

%
%one uses synchronous coupling which lead to Wasserstein bounds and from $t$ to 
%For $\ell\in\Cloc^{H-}(\R_+,\R^d)$, we will denote by $\Phi_t(\ell)$ a  solution to
%\begin{equation}\label{eq:innovation_sde}
%  \Phi_t(\ell)=\ell(t)+\int_0^t b\big(\Phi_s(\ell)\big)\,ds+\sigma\tilde{B}_t,\qquad t\geq 0,
%\end{equation}
%
%}
\paragraph{Proof of~\cref{prop:TVfrac}}
The strategy of proof is completely similar to that of~\cref{thm:fractrao} but in the simpler case where the pasts before time $0$ are equal. More precisely, in this case,~\eqref{eq:tvcondpeiozp} is replaced by:
\begin{equation}
\|{\cal L}(X_t)-\nu\|_{TV}\le \int \|{\cal L}(\Phi_{t}(\ell(x, w))-{\cal L}(\Phi_{t}(\ell(y, w))\|_{TV}\Pi(dy,d{w}).
\end{equation}
Thus, inserting the bound~\eqref{eq:TVbounddklqjq} in the special case where $w=\tilde{w}$ leads to an exponential bound. For more details, we refer to the original proof by~\cite{sieber-li}.

\paragraph{Proof of~\cref{thm:specou}}
Let $(Y_t)_{t\ge0}$ denote a stationary solution. Then, for each $t\ge0$, $Y_t$ is centered and  ${\rm Var}(Y_t)=\sigma_H^2$. Furthermore, by~\cite[Theorem 2.3]{cheridito03}, for any $0\le s\le t$, with $t-s\ge 1$, 
$${\rm Cov}(Y_s,Y_t)=\frac{\sigma^2H(2H-1)}{\lambda^2} (t-s)^{2H-2}+{\cal O}((t-s)^{2H-4}).$$
In particular, there exists $C_1$ such that for any $(s,t)$ with $t-s\ge C_1$,
$${\rm Cov}(Y_s,Y_t)\le C (t-s)^{2H-2}\le 1/2.$$
 By~\cref{lem:gaussian}, for any bounded Borel function $f$,
$$|{\rm Cov}(f(Y_s),f(Y_t))|\le 2 {\rm Var}_\nu(f) \sigma_H^{-2} |t-s|^{2H-2}.$$
Thus, if $H<1/2$, for any $t\ge 1$,
\begin{align*}
\E[|\nu_t(A)-\nu(A)|^2]&\le \frac{C}{t^2}\int_0^t \int_s^{t} {\rm Cov}({\bf 1}_A(Y_s),{\bf 1}_A(Y_u))du ds\\
&\le  \frac{C{\rm Var}_\nu({\bf 1}_A)}{t^2}\int_0^t \int_s^{t} 1\wedge |u-s|^{2H-2} du ds\le \frac{C\nu(A)}{{t^{(2-2H)\wedge 1}}}.
\end{align*}
As a consequence, $\AUNgamma$ holds with $\pi_t=\nu$, $\beta=1/2$ and $\gamma=\frac{1}{2}$ if $H<1/2$ and $\gamma=1-H$ if $H>1/2$.
\hspace{7cm} \hfill $\Box$

\begin{lem}\label{lem:gaussian}
Let $(U,V)$ be an $\ER^2$-valued centered Gaussian variable such that $\sigma_U=\sigma_V$. Then, if $|{\rm Cov}(U,V)|\le 1/2$, we have for every Borel functions $f:\ER\mapsto\ER$ and $g:\ER\mapsto\ER$ such that $\E[f^2(U) ]\vee\E[g^2(U)]<+\infty$,
$$|{\rm Cov}(f(U),g(V))|\le 2 \sqrt{{\rm Var}(f(U)){\rm Var}(g(U))} \sigma_U^{-2} |{\rm Cov}(U,V)|.$$
\end{lem}
\begin{proof}
%[Proof of~\cref{lem:gaussian}]
First, note that if the statement is true in the case $\sigma_U=\sigma_V=1$, then setting $\tilde{f}(u)=f(\sigma_U u)$, $\tilde{g}(u)=g(\sigma_U u)$, $\tilde{U}=\sigma_U^{-1}U$ and $\tilde{V}=\sigma_V^{-1}V$, we get
\begin{align*}
{\rm Cov}(f(U),g(V))={\rm Cov}(\tilde{f}(\tilde{U}),\tilde{g}(\tilde{V}))&\le 2 \sqrt{{\rm Var}(\tilde{f}(\tilde{U})){\rm Var}(\tilde{g}(\tilde{U}))} {\rm Cov}(\sigma_U^{-1}U,\sigma_V^{-1}V)\\
&\le 2\sqrt{{\rm Var}(f(U)){\rm Var}(g(U))}  \sigma_U^{-2}{\rm Cov}(U,V).
\end{align*}
We thus now assume that $\sigma_U=\sigma_V=1$. Let $(\bar{H}_k)_{k\ge1}$ denote the Hermite orthonormal basis of $\mathbb{L}^2(\ER,\gamma)$ where $\gamma$ denotes the standard normal distribution. For any $h\in \mathbb{L}^2(\ER, \gamma)$,  $ h-\int h d\gamma=\sum_{k\ge1} \langle h, \bar{H}_k\rangle_\gamma \bar{H}_k$ so that
$$
{\rm Cov}(f(U),g(V))=\sum_{k,\ell\ge1} \langle f-\int f d\gamma, \bar{H}_k\rangle_\gamma\langle g-\int g d\gamma, \bar{H}_\ell \rangle_\gamma \E[\bar{H}_k(U)\bar{H}_\ell(V)].
$$
By~\cite[Corollary 8.1.4]{peccati-taqqu} ({which is written for $H_k=\sqrt{k!} \bar{H}_k$}),  we deduce that
$$
{\rm Cov}(f(U),g(V))=\sum_{k\ge1}\langle f-\int f d\gamma, \bar{H}_k\rangle_\gamma\langle g-\int g d\gamma, \bar{H}_k\rangle_\gamma{\rm Cov}(U,V)^k.
$$
The result follows from Cauchy-Schwarz inequality and the fact that $|{\rm Cov}(U,V)|\le 1/2$.
\end{proof}

 \noindent \textbf{Acknowledgements.}  The first author benefited for this research of the support of the ``Chaire Risques Financiers'', Fondation du Risque.
 The second author benefited of the support
the Henri Lebesgue Center (ANR-11-LABX-0020-01) and the ANR
project RAWABRANCH (ANR-23-CE40-0008). The second author also thanks Julian Sieber for fruitful discussions (about the application to stationary Gaussian processes).

\small
\bibliographystyle{alpha}
\bibliography{bib_VTergo}
\end{document}